\documentclass[10pt]{amsart}
\usepackage[margin=1in]{geometry}
\usepackage{refcheck}
\usepackage{color}
\usepackage{amssymb}
\usepackage{mathtools}
\usepackage{subcaption}
\newtheorem{theorem}{Theorem}[section]
\newtheorem{lemma}[theorem]{Lemma}
\newtheorem{proposition}{Proposition}[section]

\theoremstyle{definition}

\theoremstyle{remark}
\newtheorem{remark}[theorem]{Remark}

\numberwithin{equation}{section}

\usepackage{graphicx}
\DeclareGraphicsExtensions{.pdf,.png,.jpg}
\begin{document}
		\title[Relativistic BGK model for reactive gas mixtures]{Relativistic BGK model for reactive gas mixtures}

\author[S.-Y. Cho]{Seung-Yeon Cho}
\address{Department of Mathematics
Gyeongsang National University, Jinju 52828, Republic of Korea}
\email{chosy89@gnu.ac.kr}
		
\author[B.-H. Hwang]{Byung-Hoon Hwang}
\address{Department of Mathematical education, Sangmyung University, 20, Hongjimun 2-gil, Jongno-gu, Seoul, Republic of Korea}
\email{bhh0116@smu.ac.kr}

\author[M-.S Lee ]{Myeong-Su Lee}
\address{Research Institute of Mathematics, Seoul National University, Seoul, Republic of Korea}
\email{msl3573@snu.ac.kr}

\author[S.-B. Yun]{Seok-Bae Yun}
\address{Department of Mathematics, Sungkyunkwan University, Suwon 440-746, Republic of Korea}

\email{sbyun01@skku.edu}

\begin{abstract}

We propose a BGK-type kinetic model for relativistic reactive gas mixtures. This model serves as a computationally tractable yet physically consistent alternative to the corresponding Boltzmann equation. The relaxation operator is constructed to ensure that the model correctly satisfies the conservation laws and relaxes to the proper equilibrium: a J\"{u}ttner distribution characterized by a common temperature, velocity, and chemical potentials that obey the law of mass action. Furthermore, we prove that the model satisfies an H-theorem with the same entropy functional as the original Boltzmann equation. Finally, numerical simulations are presented, which confirm that the model preserves the conserved quantities and exhibits entropy decay towards the proper J\"{u}ttner equilibrium.
\end{abstract}
\maketitle
\tableofcontents

\section{Introduction}
The kinetic theory of gases provides a fundamental framework for understanding the behavior of rarefied gas systems from a mesoscopic perspective. Many early models assume a single species undergoing purely elastic collisions. In practical scenarios, however, most gases are multi-species, and particle collisions often involve inelastic chemical reactions that exchange number densities between species. To capture these effects, researchers have extended the classical Boltzmann equation and related kinetic models to chemically reacting mixtures, resulting in extensive theoretical and numerical studies \cite{eu1977kinetic,gorecki1991molecular,nowakowski1996solution,zamlutti1997effects,rossani1999note,groppi1999kinetic,lin2017multi,sarna2021moment}. Moreover, in extreme environments such as relativistic jets, gamma-ray bursts, and high-energy-density plasmas, relativistic corrections become indispensable. Consequently, kinetic models are reformulated to be consistent with relativity \cite{rieger2002particle,synge1957relativistic,scheck2002does,coughlin2015viscous,neitz2014electron,tzoufras2013multi}. In particular, relativistic Boltzmann equations have served as the fundamental tools for analyzing thermodynamic behavior and transport phenomena in high-speed gas systems \cite{hermens1972relativistic,cercignani2002relativistic,fotakis2022multicomponent,kremer2014theory,van2016cosmological,vereshchagin2017relativistic,zhang2020towards}.

In this paper, we consider a gas system composed of four different species, denoted by $G_i\ (i=1,\cdots,4)$, undergoing a reversible bi-molecular reaction represented symbolically by $G_1+G_2\leftrightarrow G_3+G_4$. This has been widely considered in both classical and relativistic frameworks \cite{groppi1999kinetic,rossani1999note,cercignani2002relativistic}.
In the framework of relativity, the dynamics of such reactive gas mixtures can be described by the following relativistic Boltzmann equation for reactive mixtures:
\begin{equation}\label{Boltzmann}
  \partial_t f_i 
  + \frac{cp_i}{p_i^0}\!\cdot\!\nabla_x f_i
  = C_i,\qquad (i=1,\cdots,4),
\end{equation}
where $f_i$ is the momentum distribution function of the $i$-th species and  $C_i$ is the collision operator representing elastic collisions and reactive interactions between the $i$-th species and all  particles.   A detailed description of its structure and physical properties is provided in Section~\ref{sec-preli}. We note that Boltzmann-type equations, including the aforementioned relativistic Boltzmann equation, provide accurate descriptions of non-equilibrium behaviors in rarefied gas systems. They naturally satisfy fundamental physical principles, such as conservation laws for particle number (or particle four-flows in relativistic cases), momentum, total energy (or the energy-momentum tensor), as well as the H-theorem \cite{cercignani2002relativistic, cercignani1988boltzmann}. However, despite their fundamental significance and physical consistency, Boltzmann-type equations are notoriously challenging to solve numerically, primarily due to the inherent complexity of the collision operators which involve high-dimensional integrals. This computational challenge is difficult even for the classical Boltzmann equation, and obviously, it's much more severe for reactive relativistic mixtures.

To overcome the difficulty, various alternatives have been developed to replace the complicated Boltzmann-type equations. The Bhatnagar-Gross-Krook (BGK) relaxation-time approximation is one of the most well-known and representative alternatives. As an approximation of the classical Boltzmann equation, the first BGK model was proposed in 1954 under idealized physical assumptions \cite{bhatnagar1954model,welander1954temperature}. Subsequently, like many other kinetic models, the BGK approach has been progressively generalized and extended to more complex physical contexts, providing simplified counterparts to generalized Boltzmann-type equations involving multi-species particles, chemical reactions, and relativistic effects. In fact, we can find numerous studies conducted separately for classical reactive gas mixtures and relativistic gases (see Subsection \ref{subsec1-1} for detailed lists). However, to the best of our knowledge, the BGK-type models have yet to be extended to the system of relativistic reactive gas mixtures.

In this regard, this work aims to propose  a BGK-type model for relativistic reactive gas mixtures, by replacing the  Boltzmann collision operator $C_i$ of \eqref{Boltzmann} with a BGK-type relaxation operator $\widetilde{C}_i$: 
\begin{align*}
	\widetilde{C}_i=\frac{cm_i}{\tau_i p_i^0}(\mathcal{J}_i-f_i), \qquad (i=1,\cdots,4),
\end{align*}
where $\tau_i$ represents the characteristic relaxation time and the attractor $\mathcal{J}_i$ is given by the J\"{u}ttner distribution
\begin{equation}\label{juttner}
\mathcal{J}_i=\frac{g_{s_{i}}}{h^3}\exp\left({\widetilde{\beta}\widetilde{\mu}_i-\widetilde{\beta}\widetilde{U}^\mu p_{i \mu }}\right)\qquad \text{with}\quad \widetilde{\beta}:=1/k\widetilde{T}
\end{equation}
with the speed of light $c$, Planck constant $h$, Boltzmann constant $k$, and degeneracy factor $g_{s_i}$ of the $i$-th species. The auxiliary chemical potentials $\widetilde{\mu}_i$, four-velocity $\widetilde{U}^\mu$, and temperature $\widetilde{T}$ are functions of time $t$ and position $x$, which are determined in a way that the resulting model satisfies fundamental physical properties. 

This approach builds upon the existing work for classical reactive gas mixtures \cite{groppi2009kinetic} and relativistic inert gas mixtures \cite{hwang2024relativistic}, but the determination of auxiliary parameters involves non-trivial steps that go significantly beyond merely combining existing results. In the case of a relativistic inert mixture, as studied in \cite{hwang2024relativistic}, the auxiliary parameters are determined by imposing two distinct sets of constraints. First, to ensure the conservation of the particle four-flow for each species, the following species-wise cancellation property must hold:
\begin{align}
\begin{split}\label{cancell}
&\int_{\mathbb{R}^{3}}\widetilde C_{i}\ dp_{i}
= 0,\quad i=1,\cdots,4.
\end{split}
\end{align}
Second, the conservation of the total energy-momentum tensor requires that:
\begin{align}
\begin{split}\label{cancell2}
& \sum_{i=1}^4\int_{\mathbb{R}^3}p_i^\mu\widetilde C_i\ dp_i
= 0.
\end{split}
\end{align}
Condition \eqref{cancell} provides four constraints (one for each species), while \eqref{cancell2} provides another four (one for each component of the four-vector). Together, they yield a total of \emph{eight} independent constraints for determining the auxiliary parameters. By contrast, when the reversible reaction $G_{1}+G_{2}\leftrightarrow G_{3}+G_{4}$ is allowed, the conservation of the partial particle four-flow of each species no longer holds, while the conservation of the total energy-momentum tensor still holds. Instead, the relations~\eqref{cancell} is replaced by the following pair-wise cancellations:
\begin{equation}\label{paircancell}
\int_{\mathbb R^{3}}\widetilde C_{1}\,dp_{1}+
\int_{\mathbb R^{3}}\widetilde C_{3}\,dp_{3}=0,\qquad
\int_{\mathbb R^{3}}\widetilde C_{1}\,dp_{1}+
\int_{\mathbb R^{3}}\widetilde C_{4}\,dp_{4}=0,\qquad
\int_{\mathbb R^{3}}\widetilde C_{2}\,dp_{2}+
\int_{\mathbb R^{3}}\widetilde C_{3}\,dp_{3}=0.
\end{equation}
Thus, together with \eqref{cancell2}, we obtain only \emph{seven} independent constraints from the conservation laws in the case of the relativistic  reactive gas mixtures. Since the number of unknown auxiliary parameters is \emph{eight}, the system is now under-determined. To close the system, we supplement \eqref{cancell2} and \eqref{paircancell} with the mass-action law (often called the law of chemical equilibrium):
\begin{align*}\label{action}
\widetilde\mu_{1}+\widetilde\mu_{2}=\widetilde\mu_{3}+\widetilde\mu_{4},    
\end{align*} 
ensuring that the number of equations equals the number of unknowns. This new coupled system of constraints, however, causes a significant interplay among the auxiliary parameters of different species, making their determination far more intricate than in the relativistic inert case. While a similar coupled system of constraints for auxiliary parameters is, in fact, also seen in classical reactive mixtures \cite{groppi2009kinetic}, the relativistic effects cause difficulties that were not observed in previous work. Specifically, the whole determination problem can be reduced to solving the following system of relations:
\begin{equation}\label{system}
\begin{dcases}
\begin{aligned}
    &\bullet\quad \frac{g_{s_i}}{h^3}e^{\widetilde{\beta} \widetilde{\mu}_i}\int_{\mathbb{R}^3}e^{-c\widetilde{\beta} p_i^0}\,\frac{dp_i}{p_i^0}=\int_{\mathbb{R}^3}f_i\frac{dp_i}{p^0_i}+\frac{\nu_1}{\nu_2}\left(\frac{g_{s_1}}{h^3}e^{\widetilde{\beta} \widetilde{\mu}_1}\int_{\mathbb{R}^3}e^{-c\widetilde{\beta} p_1^0}\,\frac{dp_1}{p_1^0}-\int_{\mathbb{R}^3}f_1\frac{dp_1}{p_1^0}\right)\qquad \text{for}\quad i=2,3,4,
    \end{aligned}
    \\[2ex]
\begin{aligned}
&\bullet\quad\nu_1\left(\frac{g_{s_i}
}{h^3}e^{\widetilde{\beta}\widetilde{\mu}_1}\int_{\mathbb{R}^3}e^{-c\widetilde{\beta}p_1^0}\frac{dp_1}{p_1^0}-\int_{
\mathbb{R}^3}f_1\frac{dp_1}{p_1^0}\right)\\
&\qquad\qquad\times\left(\frac{\int_{\mathbb{R}^3}e^{-c\widetilde{\beta}p_1^0}dp_1}{\int_{\mathbb{R}^3}e^{-c\widetilde{\beta}p_1^0}\frac{dp_1}{p_1^0}}+\frac{\int_{\mathbb{R}^3}e^{-c\widetilde{\beta}p_2^0}dp_2}{\int_{\mathbb{R}^3}e^{-c\widetilde{\beta}p_2^0}\frac{dp_2}{p_2^0}}-\frac{\int_{\mathbb{R}^3}e^{-c\widetilde{\beta}p_3^0}dp_3}{\int_{\mathbb{R}^3}e^{-c\widetilde{\beta}p_3^0}\frac{dp_3}{p_3^0}}-\frac{\int_{\mathbb{R}^3}e^{-c\widetilde{\beta}p_4^0}dp_4}{\int_{\mathbb{R}^3}e^{-c\widetilde{\beta}p_4^0}\frac{dp_4}{p_4^0}}\right)\\
&\qquad=\frac{1}{c}\left[\left(\sum_{i=1}^{4}\nu_i\int_{\mathbb{R}^3}p_i^\mu f_idp_i\right)\left(\sum_{i=1}^{4}\nu_i\int_{\mathbb{R}^3}p_{i\mu} f_idp_i\right)\right]-\sum_{i=1}^4\nu_i\frac{\int_{\mathbb{R}^3}e^{-c\widetilde{\beta}p_i^0}dp_i}{\int_{\mathbb{R}^3}e^{-c\widetilde{\beta}p_i^0}\frac{dp_i}{p_i^0}}\int_{\mathbb{R}^3}f_i\frac{dp_i}{p_i^0}
\end{aligned}
\\[2ex]
\begin{aligned}
\bullet\quad\widetilde{\mu}_1+\widetilde{\mu}_2=\widetilde{\mu}_3+\widetilde{\mu}_4
\end{aligned}
\end{dcases}
\end{equation}
Establishing the solvability of this highly nonlinear transcendental system is a non-trivial task. Indeed, the proof requires a careful analysis of the asymptotic behavior of nonlinear functions associated with the modified Bessel function of the second kind.
\\

\subsection{Related works}\label{subsec1-1}

For the case of classical inert  gas mixtures, numerous attempts made over several decades have resulted in models  \cite{gross1956model,hamel1965kinetic,garzo1989kinetic} in 1950s--80s. However, these models lacked some essential physical properties. The first consistent BGK-type model for inert gas mixtures was developed by Andries et al. \cite{andries2002consistent} in 2002. Following this pioneering work, several additional consistent BGK-type models have subsequently been proposed \cite{haack2017conservative,bobylev2018general,klingenberg2018consistent}. Building upon these inert mixture models, extensions to chemically reactive gas mixtures have also been suggested \cite{groppi2004bhatnagar,kremer2006relaxation,groppi2009kinetic}. Groppi et al. \cite{groppi2004bhatnagar} extended the framework of Andries et al. \cite{andries2002consistent} to reactive mixtures. However, their model has critical shortcomings, including the absence of a rigorous proof of the H-theorem and potential issues regarding the positivity of auxiliary temperatures. Indeed, Kim et al. \cite{kim2021positivity} show that the auxiliary temperature could become negative for certain velocity distributions. Similarly, another model proposed by Kremer and Marques \cite{kremer2006relaxation} also fails to satisfy the H-theorem. These significant issues were first resolved by Groppi et al. \cite{groppi2009kinetic}, who adopted a simplifying assumption that sets all auxiliary mean velocities and temperatures equal across species, determining their common values from constraints derived from conservation laws and mass action laws. This approach has recently been extended further to reactive mixtures involving both monatomic and polyatomic gases \cite{bisi2018bgk, bisi2021kinetic}. For mathematical and numerical studies related to BGK-type models for classical gas mixtures, we refer to \cite{klingenberg2018existence,kim2021stationary,bae2023bgk,brull2024stationary} (mathematical results) and \cite{groppi2016semi,cho2020conservative,crestetto2020kinetic,tantos2021heat} (numerical methods).

In the relativistic framework, BGK-type models were proposed by Marle \cite{marle1965modele, marle1969etablissement} and Anderson and Witting \cite{anderson1974relativistic} for single monatomic gases. Marle's model adopts the Eckart frame \cite{eckart1940thermodynamics}, whereas the Anderson–Witting model employs the Landau–Lifshitz frame \cite{landau1987fluid}. Later, a relativistic BGK-type model for polyatomic gases was developed in \cite{pennisi2018new}, utilizing a generalized J\"{u}ttner distribution to polyatomic particles which is proposed in \cite{pennisi2017relativistic}. For inert multi-species relativistic gases, a consistent generalization of the Marle-type approximation was recently introduced by Hwang et al. \cite{hwang2024relativistic}. We mention that although relativistic BGK-type models for gas mixtures have been extensively utilized to study transport coefficients and analyze various relativistic flow phenomena \cite{kremer2003relativistic, kremer2012relativistic, kremer2014diffusion}, a rigorous discussion on how the auxiliary parameters characterizing relaxation-time operators should be determined was first addressed in \cite{hwang2024relativistic}. This determination step is crucial, as it provides the foundation for any rigorous mathematical analysis or numerical simulation of the kinetic system. Indeed, once the auxiliary parameters are properly specified, the aforementioned relativistic BGK-type models have been the subject of extensive mathematical studies \cite{bellouquid2012relativistic,calvo2020global,hwang2022global,hwang2021relativistic,hwang2021stationary,hwang2019anderson,hwang2024stationary}.

\subsection{Organization} The remainder of this paper is organized as follows. In Section~\ref{sec-preli}, we briefly review some fundamental notations and preliminaries of the relativistic kinetic theory. In Section~\ref{sec-model}, we propose our BGK-type relaxation operator $\widetilde C_i$ for relativistic reactive gas mixtures. We  determine the auxiliary parameters in the relaxation operator, by solving the nonlinear system of equations \eqref{system}. We also establish the H-theorem for the resulting model. Section~\ref{sec-numeric} presents numerical simulations in spatially homogeneous setting, which illustrates conservation, entropy dissipation, and relaxation to the correct equilibrium state.

\section{Preliminary}\label{sec-preli}
In this section, we introduce some fundamental notations and preliminaries required for the subsequent development of our relativistic BGK-type model for reactive gas mixtures.

Let $m_i$ denote the rest mass of a particle of $i$-th species. Non-relativistic kinetic theory typically considers only cases where the masses of gas particles are conserved, i.e., $m_1+m_2=m_3+m_4$ \cite{groppi2009kinetic,groppi2004bhatnagar}. However, inelastic processes of relativistic gases considered in this paper allow for the possibility that the masses of particles are not conserved. In other words, the mass defect $\Delta m:=m_1+m_2-m_3-m_4$ could be non-zero. Such a non-negligible mass defect implies the presence of binding energy, as described by Einstein's mass-energy equivalence.

The state of each species $i$ is described by its momentum distribution function $f_i(x^{\mu}, p_i^{\mu})$, which represents the number density of $i$ species at the phase point $(x^{\mu}, p_i^{\mu})$, where we used the following  the space-time coordinates $x^{\mu}$ and the four-momentums $p_i^{\mu}$:
$$
x^{\mu} =(ct,x)\in \mathbb{R}_+\times \Omega_x,\qquad p_i^{\mu}=(\sqrt{(cm_i)^2+|p_i|^2},p_i)\in\mathbb{R}_+\times\mathbb{R}^3.
$$  
For each momentum distribution $f_i$, we define macroscopic quantities as follows. First, the partial particle four-flow $N_i^\mu$ and the partial energy-momentum tensor $T_i^{\mu\nu}$ are given by
$$
N_i^\mu=c\int_{\mathbb{R}^3}p^\mu_if_i \,\frac{dp_i}{p_i^0},\qquad T_i^{\mu\nu}=c\int_{\mathbb{R}^3}p_i^\mu p_i^\nu f_i \frac{dp_i}{p_i^0}.
$$
Then, the particle four-flow of gas mixtures $N^\mu$ and the energy-momentum tensor of gas mixtures $T^{\mu\nu}$ are obtained by
$$
N^\mu=\sum_{i=1}^{4}N_i^\mu, \qquad T^{\mu\nu}=\sum_{i=1}^{4}T_i^{\mu\nu}.
$$
According to the Eckart frame \cite{eckart1940thermodynamics}, $N_i^\mu$ is decomposed into
\begin{equation*} 
N_i^\mu= n_iU_i^\mu
\end{equation*}
where $n_i$ is the (macroscopic) number density, and $U_i^\mu$ is the Eckart four-velocity
\begin{align*}\begin{split}
n_i&=\frac{1}{c}\int_{\mathbb{R}^3}p^\mu_i U_{i \mu }f_i \,\frac{dp_i}{p_i^0}=\Biggl\{\biggl(\int_{\mathbb{R}^3}f_i \,dp_i\biggl)^2-\sum_{j=1}^3\biggl(\int_{\mathbb{R}^3}p^j_if_i \,\frac{dp_i}{p_i^0}\biggl)^2\Biggl\}^{\frac{1}{2}},\cr
U_i^\mu&=\frac{c}{n_i}\int_{\mathbb{R}^3}p^\mu_if_i \,\frac{dp_i}{p_i^0}.
\end{split}\end{align*}
Here, $U^\mu_i$ has a constant length in the following sense
$$
U^\mu_i U_{i \mu }=c^2,\qquad \text{and hence}\qquad U_i^\mu=\left(\sqrt{c^2+|U_i|^2},U_i\right),
$$
where we follow the Einstein summation convention:
$$
a^\mu b_\mu = a^0 b^0-\sum_{j=1}^3a^j  b^j.
$$
The entropy four-flow of gas mixtures is defined as
\begin{equation*}\label{entropy}
S^\mu= -kc\sum_{i=1}^4\int_{\mathbb{R}^3}p_i^\mu  f_i \ln\left( \frac{f_i h^3}{g_{s_i}}\right) \,\frac{dp_i}{p_i^0}.
\end{equation*}

\subsection{Boltzmann collision operator for relativistic reactive gas mixtures}\label{subsec2-1}
This subsection is devoted to briefly reviewing the precise description and crucial properties of the Boltzmann collision operator for relativistic reactive gas mixtures. 

The Boltzmann collision operator $C_i$ is decomposed into two distinct parts, elastic inert collision and inelastic reactive collision, given by $ C_i:=\sum_{j=1}^{4}Q_{ij}+J_i$. Here, the inert collision operators $Q_{ij}$ and the reactive collision operators $J_i$ are explicitly defined by:
\begin{align*}
	Q_{ij}&=\frac{1}{p^0_i}\sum_{j=1}^{4}\int(f_i(p_i'^\mu)f_j(q_j'^\mu)-f_i(p_i^\mu)f_j(q_j^\mu))F_{ij}\sigma_{ij}d\Omega\frac{dq_j}{q^0_j},\\
	J_1&=\frac{1}{p^0_1}\int(f_3(p_3'^\mu)f_4(q_4'^\mu)-f_1(p_1^\mu)f_2(q_2^\mu))F_{12}\sigma_{12}^{34} d\Omega\frac{dq_2}{q^0_2},\\
	J_2&=\frac{1}{p^0_2}\int(f_4(p_4'^\mu)f_3(q_3'^\mu)-f_2(p_2^\mu)f_1(q_1^\mu))F_{21}\sigma_{21}^{43} d\Omega\frac{dq_1}{q^0_1},\\
	J_3&=\frac{1}{p^0_3}\int(f_1(p_1'^\mu)f_2(q_2'^\mu)-f_3(p_3^\mu)f_4(q_4^\mu))F_{34}\sigma_{34}^{12} d\Omega\frac{dq_4}{q^0_4},\\
	J_4&=\frac{1}{p^0_4}\int(f_2(p_2'^\mu)f_1(q_1'^\mu)-f_4(p_4^\mu)f_3(q_3^\mu))F_{43}\sigma_{43}^{21} d\Omega\frac{dq_3}{q^0_3},
\end{align*}
where $\sigma_{ij}$ is the collision kernel and $F_{ij}$ is the invariant flux. In the following, we present fundamental properties of the Boltzmann collision operator. Crucially, the relaxation operator to be constructed in the next section is designed to satisfy these same properties.
\begin{enumerate}
    \item {\bf Conservation laws} The Boltzmann collision operators, $C_i$, satisfies the following cancellation properties: 
\begin{align*}\begin{split}
	\int_{\mathbb{R}^3} C_idp_i + \int_{\mathbb{R}^3}C_jdp_j&=0,\qquad (i,j)=(1,3),(1,4),(2,4),\\
	\sum_{i=1}^4\int_{\mathbb{R}^3}p_i^\mu C_idp_i&=0,
\end{split}\end{align*} 
which leads to the balance equations for the particle four-flow and the energy-momentum tensor:
\begin{align*}
	&\partial_\mu N_i^\mu+\partial_\mu N_j^\mu=0,\qquad (i,j)=(1,3),\ (1,4),\ (2,4),\\
	&\sum_{i=1}^4\partial_\nu T_i^{\mu\nu}=0.
\end{align*}
\item {\bf H-theorem} The following inequality holds:
\begin{align}\label{htheorem}
	\sum_{i=1}^4\int_{\mathbb{R}^3}C_i\ln\left(\frac{f_ih^3}{g_{s_i}}\right)dp_i&\leq0,
\end{align}
which implies the entropy inequality:
\begin{align*}
		\partial_\mu S^\mu\geq 0.
\end{align*}
\item {\bf Equilibria} At equilibrium meaning that the equality holds in \eqref{htheorem}, the momentum distributions $f_i$ are J\"utnner distributions with a common temperature $1/k\beta_E$ and four-velocity $U_E^\mu$ as below
\begin{align*}
	f_i=\frac{g_{s_i}}{h^3}\exp\left(\beta_E\mu_{E_i}-\beta_E U_E^\mu p_{i\mu}\right),
\end{align*}
with the relation for the chemical potentials:
\begin{align*}
	\mu_{E_1}+\mu_{E_2}=\mu_{E_3}+\mu_{E_4}.
\end{align*}
\end{enumerate}
For more details on relativistic kinetic theory for reactive gas mixtures, we refer the reader to Chapter 7 in \cite{cercignani2002relativistic}.

\section{Relaxation Operator}\label{sec-model}
In this section, we propose a single relaxation operator to replace the Boltzmann collision operator $C_i$, based on existing relativistic BGK-type models introduced in \cite{hwang2024relativistic, marle1965modele, marle1969etablissement}. Specifically, we adopt a Marle-type relaxation operator of the form:
\begin{align*}
\widetilde{C}_i:=\frac{cm_i}{\tau_i p_i^0}\left(\mathcal{J}_i - f_i\right), \qquad i=1,\cdots,4,
\end{align*}
where $\tau_i$ denotes the characteristic relaxation time, typically of the order of the mean free time between collisions.

To ensure that the equilibrium state of our relaxation operator coincides with that of the relativistic Boltzmann equation for gas mixtures discussed in the previous section, we choose the attractor $\mathcal{J}_i$ as the J\"{u}ttner-type distribution:
\begin{align}\label{Juttner}
\mathcal{J}_i(\widetilde{\mu}_i,\widetilde{U},\widetilde{\beta}; p_i)=\frac{g_{s_{i}}}{h^3}\exp\left(\frac{\widetilde{\mu}_i}{k\widetilde{T}}-\frac{\widetilde{U}^\mu p_{i \mu }}{k\widetilde{T}}\right),
\end{align}
where the auxiliary parameters $\widetilde\mu_i$ $(i=1,\cdots,4)$, $\widetilde{U}^\mu$, and $\widetilde{T}$ represent the chemical potentials,  common four-velocity, and  common temperature  respectively. For notational convenience, we often use the notation $
\widetilde{\beta}:=\frac{1}{k\widetilde{T}}.
$ Here, the common four-velocity $\widetilde{U}^\mu$ has a constant length in the following sense:
\[
\widetilde{U}^\mu \widetilde{U}_\mu=c^2,
\qquad \text{and hence} \qquad \widetilde{U}^0=\sqrt{c^2+ |\widetilde{U} |^2}.
\]
In addition, the chemical potentials $\widetilde{\mu}$ must be chosen to satisfy the mass action law:
\begin{align}\label{eq2-2}
\widetilde\mu_1+\widetilde\mu_2=\widetilde\mu_3+\widetilde\mu_4.
\end{align}
Such construction of $\mathcal{J}_i$ in \eqref{Juttner} naturally drives the momentum distributions $f_i$ toward J\"{u}ttner distributions characterized by a common four-velocity, common temperature, and chemical potentials satisfying the mass action law. 

To further ensure that the resulting relaxation operator shares the same conservation properties as the Boltzmann collision operator discussed in Subsection \ref{subsec2-1}, we determine the auxiliary parameters by imposing the following constraints:
\begin{align}\label{eq2-1}
\begin{split}
\int_{\mathbb{R}^3} \widetilde{C}_i\,dp_i + \int_{\mathbb{R}^3}\tilde{C}_j\,dp_j&=0,\qquad (i,j)=(1,3),(1,4),(2,4),\\[3mm]
\sum_{i=1}^{4}\int_{\mathbb{R}^3} p_i^\mu \widetilde{C}_i\,dp_i&=0.
\end{split}
\end{align}
In the remainder of this section, we rigorously establish the existence of auxiliary parameters satisfying the above constraints \eqref{eq2-2} and \eqref{eq2-1}, ensuring the physical consistency of the proposed relaxation operator $\widetilde{C}_i$ with the Boltzmann operator $C_i$. 

Before proceeding further, we define some notations for clarity of presentation:
\begin{align}\label{notations}\begin{split}
&M_i(\widetilde{\beta})=\int_{\mathbb{R}^3}e^{-c\widetilde{\beta} p_i^0}\,dp_i,\quad \widetilde{M}_i(\widetilde{\beta})=\int_{\mathbb{R}^3}e^{-c\widetilde{\beta} p_i^0}\,\frac{dp_i}{p_i^0}\quad A_i(\widetilde{\beta},\widetilde{\mu}_i)=\frac{g_{s_i}}{h^3}e^{\widetilde{\beta} \widetilde{\mu}_i},\cr 
&\xi(\widetilde{\beta})=\frac{M_1(\widetilde{\beta})}{\widetilde{M}_1(\widetilde{\beta})}+\frac{M_2(\widetilde{\beta})}{\widetilde{M}_2(\widetilde{\beta})}-\frac{M_3(\widetilde{\beta})}{\widetilde{M}_3(\widetilde{\beta})}-\frac{M_4(\widetilde{\beta})}{\widetilde{M}_4(\widetilde{\beta})},\qquad \nu_i=\frac{cm_i}{\tau_i},\cr
&\Sigma(\widetilde{\beta})=\sum_{i=1}^4\frac{M_i(\widetilde{\beta})}{\widetilde{M}_i(\widetilde{\beta})}\nu_i\int_{\mathbb{R}^3}f_i\,\frac{dp_i}{p_i^0},\qquad Z=\frac{1}{c}\left[\left(\sum_{i=1}^4\nu_in_iU_i^\mu\right)\left(\sum_{i=1}^4\nu_in_iU_{i\mu}\right)\right]^{\frac 12}.
\end{split}\end{align}
Keeping this notational setup in mind, we now reformulate the constraints in \eqref{eq2-1} into a more tractable set of algebraic relations among the auxiliary parameters.
\begin{proposition}\label{JF1}
Let $f_i(p^\mu)\geq0$ $(i=1,\cdots,4)$ be integrable and not identically zero
in the almost-everywhere sense so that $N^\mu_i$ and $T^{\mu\nu}_i$ exist. Then, the system of equations \eqref{eq2-1} leads to
\begin{enumerate}
	\item The auxiliary four-velocity $\widetilde{U}^\mu$ is given by
		\begin{align*}
	\widetilde{U}^\mu=\frac{1}{Z}\sum_{i=1}^4\nu_in_iU_i^\mu.
	\end{align*}
	\item Once $\widetilde{\beta}$ and $\widetilde{\mu}_1$ are given, the chemical potentials $\widetilde\mu_i$ ($i=2,3,4$) are given by the following relations:
	\begin{align}
		\begin{split}\label{ma}
			\widetilde{M}_2A_2&=\int_{\mathbb{R}^3}f_2\frac{dp_2}{p^0_2}+\frac{\nu_1}{\nu_2}\left(\widetilde{M}_1A_1-\int_{\mathbb{R}^3}f_1\frac{dp_1}{p_1^0}\right),\\
			\widetilde{M}_3A_3&=\int_{\mathbb{R}^3}f_3\frac{dp_3}{p^0_3}-\frac{\nu_1}{\nu_3}\left(\widetilde{M}_1A_1-\int_{\mathbb{R}^3}f_1\frac{dp_1}{p_1^0}\right),\\
			\widetilde{M}_4A_4&=\int_{\mathbb{R}^3}f_4\frac{dp_4}{p^0_4}-\frac{\nu_1}{\nu_4}\left(\widetilde{M}_1A_1-\int_{\mathbb{R}^3}f_1\frac{dp_1}{p_1^0}\right).
		\end{split}
	\end{align}
	\item The chemical potential $\widetilde\mu_1$ and the auxiliary parameter $\widetilde{\beta}$ satisfy the relation:
	\begin{align}\label{ma2}
		\nu_1\left(\widetilde{M}_1A_1-\int_{\mathbb{R}^3}f_1\frac{dp_1}{p_1^0}\right)\xi=Z-\Sigma.
	\end{align}
\end{enumerate}

\end{proposition}
	\begin{proof}
		
$\bullet$ Proof of (1):	The second line of \eqref{eq2-1} gives 
\begin{align}\label{1}
\sum_{i=1}^4\nu_iA_i(\widetilde{\beta})\int_{\mathbb{R}^3}p_i^\mu e^{-\widetilde{\beta}\widetilde{U}^\mu p_{i\mu}}\frac{dp_i}{p_i^0}=\sum_{i=1}^4\frac{\nu_i}{c}n_iU_i^\mu.
\end{align}
Let $\Lambda$ be a Lorentz transformation which maps $\widetilde{U}^\mu$ into the local rest frame $(c,0,0,0)$ (see \cite{hwang2024relativistic,strain2010global}). Applying the change of variable $\Lambda p_i^\mu=P_i^\mu$, we obtain
\begin{align*}
\int_{\mathbb{R}^3}p_i^\mu e^{-\widetilde{\beta}\widetilde{U}^\mu p_{i\mu}}\frac{dp_i}{p_i^0}&=\Lambda^{-1}\int_{\mathbb{R}^3}P_i^\mu e^{-c\widetilde{\beta} P_{i}^0}\frac{dP_i}{P_i^0}\cr 
&=\Lambda^{-1}\left( \int_{\mathbb{R}^3} e^{-c\widetilde{\beta} P_{i}^0}\,dP_i ,0,0,0\right)\cr 
&=\frac{1}{c}\left(\int_{\mathbb{R}^3}e^{-c\widetilde{\beta} p_i^0}dp_i\right)\widetilde{U}^\mu
\end{align*} 
due to the oddness and the fact that 
\begin{align*}
(m_ic)^2=p^\mu_i p_{i\mu}=P^\mu_i P_{i\mu}, \quad\mbox{and hence}\quad P_i^0=\sqrt{(m_ic)^2+|P_i|^2}.
\end{align*}
This, combined with \eqref{1}, leads to
\begin{align}\label{uu}
\widetilde{U}^\mu=\left(\sum_{i=1}^4\nu_iA_i(\widetilde{\beta})M_i(\widetilde{\beta})\right)^{-1}\left(\sum_{i=1}^4\nu_in_iU_i^\mu\right).
\end{align}
By using the fact that $\widetilde{U}^\mu\widetilde{U}_\mu=c^2$, we obtain
\begin{align}\label{Z}
\sum_{i=1}^4\nu_iA_i(\widetilde{\beta})M_i(\widetilde{\beta})=\frac{1}{c}\left[\left(\sum_{i=1}^4\nu_in_iU_i^\mu\right)\left(\sum_{i=1}^4\nu_in_iU_{i\mu}\right)\right]^{1/2},
\end{align}
so that \eqref{uu} can be rewritten as
$$
	\widetilde{U}^\mu=\frac{1}{Z}\sum_{i=1}^4\nu_in_iU_i^\mu,
	$$
where we used
\begin{align*}
    Z:=\frac{1}{c}\left[\left(\sum_{i=1}^4\nu_in_iU_i^\mu\right)\left(\sum_{i=1}^4\nu_in_iU_{i\mu}\right)\right]^{1/2}.
\end{align*}
$\bullet$ Proof of (2):	The first line of \eqref{eq2-1} yields
\begin{align}\label{first line}
\frac{g_i}{h^3}e^{\widetilde{\beta} \widetilde{\mu}_i}\int_{\mathbb{R}^3}e^{-\widetilde{\beta}\widetilde{U}^\mu p_{i\mu}}\frac{dp_i}{p_i^0}+\frac{g_j}{h^3}e^{\widetilde{\beta} \tilde\mu_j}\int_{\mathbb{R}^3}e^{-\widetilde{\beta}\widetilde{U}^\mu p_{j\mu}}\frac{dp_j}{p_j^0}=\nu_i\int_{\mathbb{R}^3} 	f_i\frac{dp_i}{p_i^0}+\nu_j\int_{\mathbb{R}^3} f_j\frac{dp_j}{p_j^0}.
\end{align}
Using the change of variable $P_i^\mu=\Lambda p_i^\mu$ again,  we obtain
\begin{align*}
\int_{\mathbb{R}^3}e^{-\widetilde{\beta}\widetilde{U}^\mu p_{i\mu}}\frac{dp_i}{p_i^0}=\int_{\mathbb{R}^3}e^{-c\widetilde{\beta} P_i^0}\frac{dP_i}{P_i^0}=\widetilde{M}_i(\widetilde{\beta}),
\end{align*}
where we used the fact that (1) the Lorentz inner product is invariant under $\Lambda$, so  $\Lambda \widetilde{U}^\mu \Lambda p_{i\mu}=\widetilde{U}^\mu p_{i\mu}$, and (2) the volume element $dp_i/p_i^0$ is also invariant under $\Lambda$.	Thus \eqref{first line} can be rewritten as
\begin{align}\label{first 2}
\nu_iA_i(\widetilde{\beta})\widetilde{M}_i(\widetilde{\beta})+\nu_jA_j(\widetilde{\beta})\widetilde{M}_j(\widetilde{\beta})=\nu_i\int_{\mathbb{R}^3} 	f_i\frac{dp_i}{p_i^0}+\nu_j\int_{\mathbb{R}^3} f_j\frac{dp_j}{p_j^0}, \qquad (i,j)=(1,3),(1,4),(2,4).
\end{align}	
Then, it follows directly from \eqref{first 2} that
\begin{align}\label{tempma}
\begin{split}
\widetilde{M}_2A_2&=\int_{\mathbb{R}^3}f_2\frac{dp_2}{p^0_2}+\frac{\nu_1}{\nu_2}\left(\widetilde{M}_1A_1-\int_{\mathbb{R}^3}f_1\frac{dp_1}{p_1^0}\right),\\
\widetilde{M}_3A_3&=\int_{\mathbb{R}^3}f_3\frac{dp_3}{p^0_3}-\frac{\nu_1}{\nu_3}\left(\widetilde{M}_1A_1-\int_{\mathbb{R}^3}f_1\frac{dp_1}{p_1^0}\right),\\
\widetilde{M}_4A_4&=\int_{\mathbb{R}^3}f_4\frac{dp_4}{p^0_4}-\frac{\nu_1}{\nu_4}\left(\widetilde{M}_1A_1-\int_{\mathbb{R}^3}f_1\frac{dp_1}{p_1^0}\right).
\end{split}
\end{align}
\noindent$\bullet$ Proof of (3): Substituting \eqref{tempma} into \eqref{Z}, we obtain
\begin{align*}
&\nu_1\widetilde{M}_1A_1\left(\frac{M_1}{\widetilde{M}_1}+\frac{M_2}{\widetilde{M}_2}-\frac{M_3}{\widetilde{M}_3}-\frac{M_4}{\widetilde{M}_4}\right)-\nu_1\int_{\mathbb{R}^3}f_1\frac{dp_1}{p_1^0}\left(\frac{M_1}{\widetilde{M}_1}+\frac{M_2}{\widetilde{M}_2}-\frac{M_3}{\widetilde{M}_3}-\frac{M_4}{\widetilde{M}_4}\right)+\sum_{i}^4\frac{M_i}{\widetilde{M}_i}\nu_i\int_{\mathbb{R}^3}f_i\frac{dp_i}{p_i^0}\\
&=\frac{1}{c}\left[\left(\sum_{i=1}^4\nu_in_iU_i^\mu\right)\left(\sum_{i=1}^4\nu_in_iU_{i\mu}\right)\right]^{1/2},
\end{align*}
which together with \eqref{notations} gives  
\begin{align*}
	\nu_1\left(\widetilde{M}_1A_1-\int_{\mathbb{R}^3}f_1\frac{dp_1}{p_1^0}\right)\xi=Z-\Sigma.
\end{align*}
\end{proof}
We observed that the auxiliary four-velocity $\widetilde{U}^\mu$ can be explicitly determined by macroscopic quantities of $f_i$, and the chemical potentials $\widetilde{\mu}_i$ ($i=2,3,4$) can be determined uniquely once both $\widetilde{\mu}_1$ and $\widetilde{\beta}$ are given. Of course, for $\widetilde{\mu}_i$ ($i=2,3,4$) 
to be physically meaningful (i.e., real-valued), the search for pairs $(\widetilde{\mu}_1,\widetilde{\beta})$ must be confined to the specific region $\mathcal{D}$:
\begin{align*}
\mathcal{D}:=\bigg\{(\widetilde{\mu}_1,\widetilde{\beta})\in\mathbb{R}\times\mathbb{R}_+\bigg| -\frac{\nu_2}{\nu_1}\int_{\mathbb{R}^3}f_2\frac{dp_2}{p_2^0}<\widetilde{M}_1A_1-\int_{\mathbb{R}^3}f_1\frac{dp_1}{p_1^0}<\min\left\{\frac{\nu_3}{\nu_1}\int_{\mathbb{R}^3}f_3\frac{dp_3}{p_3^0},\frac{\nu_4}{\nu_1}\int_{\mathbb{R}^3}f_4\frac{dp_4}{p_4^0}\right\}\bigg\},
\end{align*} 
which guarantees that all the right hand side of \eqref{ma} are positive so that $\widetilde{\mu}_i$ $(i=2,3,4)$ can be determined as real value via \eqref{ma}. Also, we note that the notational definition of $A_i$ directly gives 
\begin{align*}
    \frac{A_1A_2}{A_3A_4}=\frac{g_{s_1}g_{s_2}}{g_{s_3}g_{s_4}}e^{\widetilde{\mu}_1+\widetilde{\mu}_2-\widetilde{\mu}_3-\widetilde{\mu}_4}.
\end{align*} 
This means that the mass action law \eqref{eq2-2} is equivalent to 
\begin{align*}
    \frac{A_1A_2}{A_3A_4}=\frac{g_{s_1}g_{s_2}}{g_{s_3}g_{s_4}},
\end{align*}
which, together with \eqref{ma}, can be rewritten as a relation for $\widetilde{\mu}_1$ and $\widetilde{\beta}$:
\begin{align}\label{mu_rel}
	\frac{\widetilde{M}_3\widetilde{M}_4}{\widetilde{M}_1\widetilde{M}_2}\frac{\widetilde{M}_1A_1\left(\int_{\mathbb{R}^3}f_2\frac{dp_2}{p^0_2}+\frac{\nu_1}{\nu_2}\left(\widetilde{M}_1A_1-\int_{\mathbb{R}^3}f_1\frac{dp_1}{p_1^0}\right)\right)}{\left(\int_{\mathbb{R}^3}f_3\frac{dp_3}{p^0_3}-\frac{\nu_1}{\nu_3}\left(\widetilde{M}_1A_1-\int_{\mathbb{R}^3}f_1\frac{dp_1}{p_1^0}\right)\right)\left(\int_{\mathbb{R}^3}f_4\frac{dp_4}{p^0_4}-\frac{\nu_1}{\nu_4}\left(\widetilde{M}_1A_1-\int_{\mathbb{R}^3}f_1\frac{dp_1}{p_1^0}\right)\right)}=\frac{g_{s_1}g_{s_2}}{g_{s_3}g_{s_4}}.
\end{align}
The problem of determining the full set of auxiliary parameters is thus reduced to finding a pair $(\widetilde{\mu}_1,\widetilde{\beta})$ within the feasible set $\mathcal{D}$ that simultaneously satisfies \eqref{ma2} and \eqref{mu_rel}. Solving this coupled system is the main challenge. To solve this system, we first establish the following lemma for the root of $Z-\Sigma$ in the right hand side of \eqref{ma2}.
\begin{lemma}$Z-\Sigma$ has the unique root for $\widetilde{\beta}>0.$
\begin{proof} 
    We observe from the definition of $Z$ and $n_iU_i^\mu$ that
	\begin{align*} 
			Z&=\frac{1}{c}\biggl\{\sum_{i,j=1}^4\nu_i\nu_j(n_iU_i^\mu)(n_jU_j^\mu)\biggl\}^{\frac 12}\cr 
			&=\left(\sum_{i,j=1}^4\nu_i \nu_j\iint_{\mathbb{R}^3\times \mathbb{R}^3}p^\mu_i q_{j\mu} f_i(p)f_j(q) \,\frac{dp_i}{p_i^0}\frac{dq_i}{q_i^0}\right)^{\frac 12}\cr 
			&\ge \left(\sum_{i,j=1}^4\nu_i \nu_j\iint_{\mathbb{R}^3\times \mathbb{R}^3}c^2m_im_j f_i(p)f_j(q) \,\frac{dp_i}{p_i^0}\frac{dq_i}{q_i^0}\right)^{\frac 12}\cr 
			&=c\sum_{i=1}^4 m_i\nu_i\int_{ \mathbb{R}^3} f_i \,\frac{dp_i}{p_i^0}, 
 \end{align*}
	where we used the Cauchy–Schwarz inequality:
	$$
	p_i^\mu q_{j\mu}=\sqrt{(cm_i)^2+|p_i|^2}\sqrt{(cm_j)^2+|q_j|^2}-p_i\cdot q_j\ge c^2m_im_j.
	$$
	Thus, we find
	$$
	Z-\Sigma \ge \sum_{i=1}^4\nu_v \left(cm_i-\frac{M_i(\widetilde{\beta})}{\widetilde{M}_i(\widetilde{\beta})} \right) \int_{ \mathbb{R}^3} f_i \,\frac{dp_i}{p_i^0}.
	$$
	On the other hand, it was shown in the proof of  \cite[Proposition 4.2]{hwang2024relativistic} that 
	\begin{align}\label{limits}
		\frac{d}{d\widetilde{\beta}}\left(\frac{M_i}{\widetilde{M}_i} \right)(\widetilde{\beta})<0,\qquad \lim\limits_{\widetilde{\beta}\rightarrow0}\left(\frac{M_i}{\widetilde{M}_i}\right)(\widetilde{\beta})=\infty,\qquad \lim\limits_{\widetilde{\beta}\rightarrow\infty}\left(\frac{M_i}{\widetilde{M}_i}\right)(\widetilde{\beta})=cm_i,
	\end{align}
	which implies that $Z-\Sigma$ is monotone and 
	\begin{align}
		\begin{split}\label{limits1}
			&\lim\limits_{\widetilde{\beta}\rightarrow0}(Z-\Sigma) =-\infty,\qquad
			\lim\limits_{\widetilde{\beta}\rightarrow\infty}(Z-\Sigma)=\frac{1}{c}\biggl\{\sum_{i,j=1}^4\nu_i\nu_j(n_iU_i^\mu)(n_jU_j^\mu)\biggl\}^{\frac 12}-c\sum_{i=1}^4\nu_vm_i \int_{ \mathbb{R}^3} f_i \,\frac{dp_i}{p_i^0} \ge 0.
		\end{split}
	\end{align}
	Therefore, $Z-\Sigma$ has a unique root by the intermediate value theorem.
\end{proof}
\end{lemma}
Let us denote the unique root of $Z-\Sigma$ as $\beta^\#$. We first assume that  $\xi(\beta^\#)\neq0$. Then, we know that for any $\widetilde{\beta}>0$ such that $\xi(\widetilde{\beta})=0$ there is no $\widetilde{\mu}_1$ satisfying \eqref{ma2}. Thus, it is enough to consider $\widetilde{\beta}$ such that $\xi\neq0$. For such $\widetilde{\beta}$, the relation \eqref{ma2} is equivalent to 
\begin{align}\label{mu_1}
\widetilde{M}_1A_1=\int_{\mathbb{R}^3}f_1\frac{dp_1}{p_1^0}+\frac{1}{\nu_1}\frac{Z-\Sigma}{\xi}.
\end{align} Also, this, together with \eqref{ma}, gives that
\begin{align}\label{mu_i}
	\widetilde{M}_iA_i&=\int_{\mathbb{R}^3}f_i\frac{dp_i}{p_i^0}\pm\frac{1}{\nu_i}\frac{Z-\Sigma}{\xi},\quad (+\ \text{for}\ i=2,\ \text{and}\ -\ \text{for}\ i=3,4).
\end{align} 
Putting \eqref{mu_1} and \eqref{mu_i} onto \eqref{mu_rel}, we have
\begin{align}\label{mu_rel2}
	\frac{\widetilde{M}_3\widetilde{M}_4}{\widetilde{M}_1\widetilde{M}_2}\cdot\frac{\left(\int_{\mathbb{R}^3}f_1\frac{dp_1}{p_1^0}+\frac{1}{\nu_1}\frac{Z-\Sigma}{\xi}\right)\left(\int_{\mathbb{R}^3}f_2\frac{dp_2}{p_2^0}+\frac{1}{\nu_2}\frac{Z-\Sigma}{\xi}\right)}{\left(\int_{\mathbb{R}^3}f_3\frac{dp_3}{p_3^0}-\frac{1}{\nu_3}\frac{Z-\Sigma}{\xi}\right)\left(\int_{\mathbb{R}^3}f_4\frac{dp_4}{p_4^0}-\frac{1}{\nu_4}\frac{Z-\Sigma}{\xi}\right)}=\frac{g_{s_1}g_{s_2}}{g_{s_3}g_{s_4}}.
\end{align} 
Since the auxiliary parameter $\widetilde{\mu}_1$ can be determined by \eqref{mu_1} once $\widetilde{\beta}$ is fixed, the entire problem is reduced to proving that there exists a $\widetilde{\beta}$ satisfying \eqref{mu_rel2} on the set $\mathcal{D}_\beta$:
\begin{align*}
\mathcal{D}_\beta:=&\mathcal{D}\cap\left\{(\widetilde{\mu}_1,\widetilde{\beta})\in\mathbb{R}\times\mathbb{R}_+ \bigg | \quad \widetilde{M}_1A_1=\int_{\mathbb{R}^3}f_1\frac{dp_1}{p_1^0}+\frac{1}{\nu_1}\frac{Z-\Sigma}{\xi}\right\}\\
=&\bigg\{\widetilde{\beta}\in\mathbb{R}_+\ \bigg| \ \max\biggl\{-\nu_1\int_{\mathbb{R}^3}f_1\frac{dp_1}{p_1^0},-\nu_2\int_{\mathbb{R}^3}f_2\frac{dp_2}{p_2^0}\bigg\}<\frac{Z-\Sigma}{\xi}(\widetilde{\beta})<\min\biggl\{\nu_3\int_{\mathbb{R}^3}f_3\frac{dp_3}{p_3^0},\nu_4\int_{\mathbb{R}^3}f_4\frac{dp_4}{p_4^0}\bigg\}\bigg\}.
\end{align*}
The following proposition rigorously establishes that such a auxiliary parameter $\widetilde{\beta}$ indeed exists.
\begin{proposition}\label{JF2} Under the same assumption in Proposition \ref{JF1}, there exists the unique solution  $\widetilde{\beta}\in\mathcal{D}_\beta$ to \eqref{mu_rel2}.
\begin{proof} We first give some computational observations. For $\xi$, we have
\begin{align}\label{xiinfty}
	\lim\limits_{\widetilde{\beta}\rightarrow\infty}\xi(\widetilde{\beta})=c(m_1+m_2-m_3-m_4)=c\Delta m.
\end{align}
where we used \eqref{limits}.
Also, we note that $M_i$ and $\widetilde{M}_i$ can be written as
\begin{align*} 
M_i(\widetilde{\beta})&= \frac{4\pi (m_ic)^2}{c\widetilde{\beta}} K_2(m_ic^2\widetilde{\beta})\\
\widetilde{M}_i(\widetilde{\beta})&=2\pi(m_ic)^2 \left\{K_2(m_ic^2\widetilde{\beta})-K_0(m_ic^2\widetilde{\beta}) \right\}\\
&=\frac{2\pi (m_ic)}{c\widetilde{\beta}}K_1(m_ic^2\widetilde{\beta})
\end{align*}
by using the change of variables $\frac{p_i}{m_ic}\rightarrow p_i$, spherical coordinates and integration by parts, where $K_i$ denotes the modified Bessel function of the second kind
\begin{align*}
K_0(\beta)&=\int_0^\infty \frac{1}{\sqrt{1+r^2}}e^{-\beta\sqrt{1+r^2}}\,dr,\cr 
K_1(\beta)&=\int_0^\infty e^{-\beta\sqrt{1+r^2}}\,dr,\cr 
K_2(\beta)&=\int_0^\infty \frac{2r^2+1}{\sqrt{1+r^2}} e^{-\beta\sqrt{1+r^2}}\,dr.
\end{align*}
By using this, we obtain 
\begin{align*}
	\frac{M_i(\widetilde{\beta})}{\widetilde{M}_i(\widetilde{\beta})}&=\frac{2}{c\widetilde{\beta}}\left\{1+\frac{K_0(m_ic^2\widetilde{\beta})}{K_2(m_ic^2\widetilde{\beta})-K_0(m_ic^2\widetilde{\beta})} \right\} \cr 
	&=\frac{2}{c\widetilde{\beta}}+m_ic \frac{K_0(m_ic^2\widetilde{\beta})}{K_1(m_ic^2\widetilde{\beta})},
\end{align*}
which gives that
\begin{equation*} 
	\xi(\widetilde{\beta})=m_1c \frac{K_0(m_1c^2\widetilde{\beta})}{K_1(m_1c^2\widetilde{\beta})}+m_2c \frac{K_0(m_2c^2\widetilde{\beta})}{K_1(m_2c^2\widetilde{\beta})}-m_3c \frac{K_0(m_3c^2\widetilde{\beta})}{K_1(m_3c^2\widetilde{\beta})}-m_4c \frac{K_0(m_4c^2\widetilde{\beta})}{K_1(m_4c^2\widetilde{\beta})}.
\end{equation*}
We note from \cite{cercignani2002relativistic} that for small values of $\beta$, the modified Bessel function of the second kind can be expanded as
\begin{align*}
	K_0(\beta)&=-\ln \beta-\gamma+\ln 2+\frac 14\beta^2\left(-\ln \beta-\gamma+1+\ln 2\right)+\mathcal{O}(\beta^3), \cr 
	K_1(\beta)&=\frac 1\beta+\frac \beta 4\left(2\ln \beta +2\gamma-1-2\ln 2\right)+\mathcal{O}(\beta^3),
\end{align*}
where $\gamma$ is some negative constant. This implies that the positive function $K_0/K_1$ approaches   $0$ as $\beta\to 0$ so that 
\begin{align}\label{xizero} 
\lim\limits_{\widetilde{\beta}\rightarrow0}\xi(\widetilde{\beta})=0.
\end{align}
Combining \eqref{limits1}, \eqref{xiinfty}, and \eqref{xizero}, we have
\begin{equation} \label{limits2}
	\lim\limits_{\widetilde{\beta}\rightarrow0}\frac{Z-\Sigma}{\xi}(\widetilde{\beta})=\pm\infty,\qquad \lim\limits_{\widetilde{\beta}\rightarrow\infty }\frac{Z-\Sigma}{\xi}(\widetilde{\beta})=C_\infty
\end{equation}
where $C_\infty$ is a positive constant defined by
\begin{align*}
	C_\infty =
	\begin{cases*}
		\infty \qquad \text{if}\ \Delta m = 0, \\
		\frac{1}{c^2\Delta m} \left[ \left\{ \sum_{i,j=1}^4 \nu_i \nu_j (n_i U_i^\mu)(n_j U_j^\mu) \right\}^{\frac{1}{2}} 
		- c^2 \sum_{i=1}^4 \nu_i m_i \int_{\mathbb{R}^3} f_i \,\frac{dp_i}{p_i^0} \right] \qquad \text{if}\ \Delta m \neq 0.
	\end{cases*}
\end{align*}
Then, for notational simplicity, we define a function $\Phi:\mathcal{D}_\beta\rightarrow\Phi(\mathcal{D}_\beta)$ by 
$$
\Phi(\widetilde{\beta}):=\frac{\widetilde{M}_3\widetilde{M}_4}{\widetilde{M}_1\widetilde{M}_2}\cdot\frac{\left(\int_{\mathbb{R}^3}f_1\frac{dp_1}{p_1^0}+\frac{1}{\nu_1}\frac{Z-\Sigma}{\xi}\right)\left(\int_{\mathbb{R}^3}f_2\frac{dp_2}{p_2^0}+\frac{1}{\nu_2}\frac{Z-\Sigma}{\xi}\right)}{\left(\int_{\mathbb{R}^3}f_3\frac{dp_3}{p_3^0}-\frac{1}{\nu_3}\frac{Z-\Sigma}{\xi}\right)\left(\int_{\mathbb{R}^3}f_4\frac{dp_4}{p_4^0}-\frac{1}{\nu_4}\frac{Z-\Sigma}{\xi}\right)}
$$
where the domain $\mathcal{D}_\beta$ is given as
\begin{align*}
	\mathcal{D}_\beta=\left\{\widetilde{\beta}\in\mathbb{R}_+\ : \ \max\biggl\{-\nu_1\int_{\mathbb{R}^3}f_1\frac{dp_1}{p_1^0},-\nu_2\int_{\mathbb{R}^3}f_2\frac{dp_2}{p_2^0}\biggl\}<\frac{Z-\Sigma}{\xi}(\widetilde{\beta})<\min\biggl\{\nu_3\int_{\mathbb{R}^3}f_3\frac{dp_3}{p_3^0},\nu_4\int_{\mathbb{R}^3}f_4\frac{dp_4}{p_4^0}\biggl\}\right\}.
\end{align*}
Here, we note that $\beta^\#$ belongs to $\mathcal{D}_\beta$, which implies that $\mathcal{D}_\beta$ is non-empty. Now, to complete the proof, we will show that (i) $\mathcal{D}_\beta$ is an interval in $\mathbb{R}_+$ where the sign of $\xi$ is unchanged,  (ii) $\Phi$ is strictly monotone on $\mathcal{D}_\beta$, and (iii) the range of $\Phi$ on $\mathcal{D}_\beta$ equals $(0,\infty)$. 
\\

\noindent(i) Interval $\mathcal{D}_\beta$: From elementary computations, we get
\begin{align}\label{monotone} \begin{split}
		\frac{d}{d \widetilde{\beta}}\left\{\frac{Z-\Sigma}{\xi}\right\} 
		=&\frac{1}{\xi}\left\{(Z-\Sigma)'-\left(\frac{Z-\Sigma}{\xi}\right)\xi'\right\}\\
		=&\frac{1}{\xi}\left\{-\sum_{i=1}^4\left(\frac{M_i}{\widetilde{M}_i}\right)'\nu_i\int_{\mathbb{R}^3}f_i\frac{dp_i}{p_i^0}-\left(\frac{Z-\Sigma}{\xi}\right)\left(\left(\frac{M_1}{\widetilde{M}_1}\right)'+\left(\frac{M_2}{\widetilde{M}_2}\right)'-\left(\frac{M_3}{\widetilde{M}_3}\right)'-\left(\frac{M_4}{\widetilde{M}_4}\right)'\right)\right\}\\
		=&\frac{1}{\xi}\bigg\{-\left(\frac{M_1}{\widetilde{M}_1}\right)'\left(\nu_1\int_{\mathbb{R}^3}f_1\frac{dp_1}{p_1^0}+\frac{Z-\Sigma}{\xi}\right)-\left(\frac{M_2}{\widetilde{M}_2}\right)'\left(\nu_2\int_{\mathbb{R}^3}f_2\frac{dp_2}{p_2^0}+\frac{Z-\Sigma}{\xi}\right)\bigg\}\\
		&+\frac{1}{\xi}\bigg\{ -\left(\frac{M_3}{\widetilde{M}_3}\right)'\left(\nu_3\int_{\mathbb{R}^3}f_3\frac{dp_3}{p_3^0}-\frac{Z-\Sigma}{\xi}\right)-\left(\frac{M_4}{\widetilde{M}_4}\right)'\left(\nu_4\int_{\mathbb{R}^3}f_4\frac{dp_4}{p_4^0}-\frac{Z-\Sigma}{\xi}\right) \bigg\}\\
		=&:\frac{K}{\xi}.
\end{split}\end{align}
By definition of $\mathcal{D}_\beta$ and \eqref{limits}, we see that $K$ is positive on $\mathcal{D}_\beta$. Let $I=(\beta_l,\beta_r)$ be an arbitrary maximal connected interval contained in $\mathcal{D}_\beta$ where $\beta_l\neq0$ due to \eqref{limits2}. Then, since $\xi$ and $Z-\Sigma$ does not vanish simultaneously, $\frac{Z-\Sigma}{\xi}$ diverges if $\xi$ gose to zero. This, together with the fact that $\xi$ is continuous in $\widetilde{\beta}$, implies that $\xi$ is sign-preserving on $I$. Thus, by \eqref{monotone}, we know that $\frac{Z-\Sigma}{\xi}$ is strictly monotone on $I$. 

To show that $\mathcal{D}_\beta$ is an interval in $\mathbb{R}_+$, we prove that such a maximal interval is unique. For $\beta_r$, there are two cases that may appear. The first case is  $\beta_r<\infty$, where the range of $\frac{Z-\Sigma}{\xi}$ in $I$ is
\begin{align*}
	\left(\max\left\{-\nu_1\int_{\mathbb{R}^3}f_1\frac{dp_1}{p_1^0},-\nu_2\int_{\mathbb{R}^3}f_2\frac{dp_2}{p_2^0}\right\},\min\left\{\nu_3\int_{\mathbb{R}^3}f_3\frac{dp_3}{p_3^0},\nu_4\int_{\mathbb{R}^3}f_4\frac{dp_4}{p_4^0}\right\}\right).
\end{align*}
The second case is $\beta_r=\infty$, i.e., $I=(\beta_l,\infty)$. We note that this case does not happen when $\Delta m=0$ where $\frac{Z-\Sigma}{\xi}$ goes to $\infty$ as $\widetilde{\beta}\rightarrow\infty$. Without loss of generality, we set $\Delta m>0$. Then, $\left(\frac{Z-\Sigma}{\xi}\right)'$ is positive for sufficiently large $\widetilde{\beta}$ due to the positivity of $K$, \eqref{xiinfty}, and \eqref{monotone}. Since $\frac{Z-\Sigma}{\xi}$ is strictly increasing on $I$, its range on $I$ is
\begin{align*}
	\left(\max\left\{-\nu_1\int_{\mathbb{R}^3}f_1\frac{dp_1}{p_1^0},-\nu_2\int_{\mathbb{R}^3}f_2\frac{dp_2}{p_2^0}\right\},C_\infty\right).
\end{align*}
For both cases, by the intermediate value theorem, $I$ has the unique root of the function $Z-\Sigma$, denoted as $\beta^\#$. This means that there couldn't exist two disjoint maximal intervals in $\mathcal{D}_\beta$. Therefore, $\mathcal{D}_\beta$ is an interval in $\mathbb{R}_+$, where the sign of $\xi$ is unchanged.
\newline

\noindent(ii) Strict monotonicity: We rewrite $\Phi(\widetilde{\beta})$ as follows:
\begin{align*}
    \Phi(\widetilde{\beta})=\frac{\widetilde{M}_3\widetilde{M}_4}{\widetilde{M}_1\widetilde{M}_2}\cdot\frac{B_1B_2}{B_3B_4},
\end{align*}
where $B_i$ denotes
$$
B_i(\widetilde{\beta})=\begin{cases}
\int_{\mathbb{R}^3}f_i\frac{dp_i}{p_i^0}+\frac{1}{\nu_i}\frac{Z-\Sigma}{\xi}(\widetilde{\beta}), &\mbox{for } i=1,2,\cr 
\int_{\mathbb{R}^3}f_i\frac{dp_i}{p_i^0}-\frac{1}{\nu_i}\frac{Z-\Sigma}{\xi}(\widetilde{\beta}), &\mbox{for }i=3,4.
\end{cases}
$$ Then, it is straightforward that
\begin{align*}
	\Phi'(\widetilde{\beta})=\Phi(\widetilde{\beta})\times\left(\frac{\widetilde{M}_3'}{\widetilde{M}_3}+\frac{\widetilde{M}_4'}{\widetilde{M}_4}-\frac{\widetilde{M}_1'}{\widetilde{M}_1}-\frac{\widetilde{M}_2'}{\widetilde{M}_2}+\frac{B_1'}{B_1}+\frac{B_2'}{B_2}-\frac{B_3'}{B_3}-\frac{B_4'}{B_4}\right)(\widetilde{\beta}).
\end{align*}
Notice that
\begin{align*}
	\frac{d\widetilde{M}_i}{d\widetilde{\beta}} =-cM_i ,\quad \mbox{and}\quad  \frac{dB_i}{d\widetilde{\beta}} =\begin{cases}
	\frac{1}{\nu_i}\left(\frac{Z-\Sigma}{\xi}\right)^\prime(\widetilde{\beta})&\mbox{for } i=1,2,\cr 
	-\frac{1}{\nu_i}\left(\frac{Z-\Sigma}{\xi}\right)^\prime(\widetilde{\beta}) &\mbox{for } i=3,4.
	\end{cases}
\end{align*}
This, together with \eqref{monotone}, gives
\begin{align*}
	\Phi'(\widetilde{\beta})&=\Phi(\widetilde{\beta})\times\left(c\xi+\left(\frac{Z-\Sigma}{\xi}\right)'\sum_{i=1}^4\frac{1}{\nu_iB_i}\right)(\widetilde{\beta})\cr 
	&=\Phi(\widetilde{\beta})\times\left(c\xi+\frac{K}{\xi}\sum_{i=1}^4\frac{1}{\nu_iB_i}\right)(\widetilde{\beta}),
\end{align*} 
which means $sign(\Phi')=sign(\xi)$ on $\mathcal{D}_\beta$. Since $\xi$ is sign-preserving on $\mathcal{D}_\beta$, we conclude that $\Phi$ is strictly monotone on $\mathcal{D}_\beta$.\\

\noindent(iii) Range of $\Phi$ in $\mathcal{D}_\beta$: We consider two cases separately.
\begin{enumerate}
	\item[Case 1:]  $\mathcal{D}_\beta=(\beta_l,\beta_r)$ with $\beta_l\neq0$ and $\beta_r<\infty$.\\ 
	If $\widetilde{\beta}$ goes to a boundary, then one of $B_i$ goes to zero or infinity, and hence the range of $\Phi$ in $\mathcal{D}_\beta$ is $(0,\infty)$.
	
	\item[Case 2:] $\mathcal{D}_\beta=(\beta_l,\infty)$ with $\beta_l\neq0$.\\
	If $\widetilde{\beta}$ goes to $\beta_l$, then $\Phi(\beta_l)=0$ since one of $B_1$ and $B_2$ goes to zero. Also, we have
	\begin{align*}
		\lim\limits_{\widetilde{\beta}\rightarrow\infty}\Phi(\widetilde{\beta})&=\lim\limits_{\widetilde{\beta}\rightarrow\infty}\frac{\widetilde{M}_3\widetilde{M}_4}{\widetilde{M}_1\widetilde{M}_2}\cdot \lim\limits_{\widetilde{\beta}\rightarrow\infty}\frac{\left(\int_{\mathbb{R}^3}f_1\frac{dp_1}{p_1^0}+\frac{1}{\nu_1}\frac{Z-\Sigma}{\xi}\right)\left(\int_{\mathbb{R}^3}f_2\frac{dp_2}{p_2^0}+\frac{1}{\nu_2}\frac{Z-\Sigma}{\xi}\right)}{\left(\int_{\mathbb{R}^3}f_3\frac{dp_3}{p_3^0}-\frac{1}{\nu_3}\frac{Z-\Sigma}{\xi}\right)\left(\int_{\mathbb{R}^3}f_4\frac{dp_4}{p_4^0}-\frac{1}{\nu_4}\frac{Z-\Sigma}{\xi}\right)}\\
		&=:L_1\times L_2.
	\end{align*}
	Then the limit $L_2$ converge to a positive constant. Instead, $L_1$ diverge as follows:
	\begin{align*}
		L_1&=\lim\limits_{\widetilde{\beta}\rightarrow\infty}\frac{\widetilde{M}_3\widetilde{M}_4}{\widetilde{M}_1\widetilde{M}_2}\\
		&=\lim\limits_{\widetilde{\beta}\rightarrow\infty}\left(\frac{m_1m_2}{m_3m_4}\right)^{1/2}e^{c\widetilde{\beta}(m_1+m_2-m_3-m_4)}\\
		&=\infty.
	\end{align*}
	where we used the asymptotic expansion of the modified Bessel function of the second kind.
\end{enumerate}
For both cases, the range of $\Phi$ is $(0,\infty)$, which completes the proof.
 \end{proof}
\end{proposition}
\begin{remark}[The non-generic case: $\xi(\beta^\#)=0$]\label{remark} In Proposition \ref{JF2}, we established the existence of a unique solution pair $(\widetilde{\mu}_1,\widetilde{\beta})\in\mathcal{D}$ to the system of \eqref{ma2} and \eqref{mu_rel} for the generic case where $\xi(\beta^\#)\neq0$. We now show that an appropriate solution pair $(\widetilde{\mu}_1,\widetilde{\beta})\in\mathcal{D}$ also exists for the non-generic case where $\xi(\beta^\#)=0$. Specifically, for $\widetilde{\beta}=\beta^\#$, the both sides of \eqref{ma2} 
    \begin{align*}
		\nu_1\left(\widetilde{M}_1A_1-\int_{\mathbb{R}^3}f_1\frac{dp_1}{p_1^0}\right)\xi=Z-\Sigma,
	\end{align*}
become zero simultaneously. Thus, \eqref{ma2} is satisfied for any choice of $\widetilde{\mu}_1$. Fixing $\widetilde{\beta}=\beta^\#$ and recalling that $A_1$ is defined as $\frac{g_{s_1}}{h^3}e^{\beta^\#\widetilde{\mu}_1}$, the left hand side of \eqref{mu_rel} can be viewed as an increasing function of $\widetilde{\mu}_1$, ranging from $0$ to $\infty$ on the domain $\mathcal{D}_\mu:=\mathcal{D}\cap\left\{\widetilde{\beta}=\beta^\#\right\}$. Therefore, by the intermediate value theorem, there exists a $\widetilde{\mu}_1$ satisfying \eqref{mu_rel}, in the domain $\mathcal{D}_\mu$. This choice of $\widetilde{\mu}_1$ and $\widetilde{\beta}$ provides a valid pair for the system of \eqref{ma2} and \eqref{mu_rel}. However, uniqueness is not guaranteed, as it is possible that a distinct solution pair of $(\widetilde{\mu}_1,\widetilde{\beta})$ exists. This scenario, however, is non-generic, as it requires the specific condition $\xi(\beta^\#) = 0$. A further investigation of conditions for uniqueness is beyond the scope of this paper.
\end{remark}

Combining the results established in Proposition \ref{JF1}, \ref{JF2}, and Remark \ref{remark}, we summarize the main result of this section in the following theorem:
\begin{theorem}\label{JF3} Under the same assumption as in Proposition \ref{JF1}, there exists a J\"{u}ttner distribution $\mathcal{J}_i$ satisfying the system of equations \eqref{eq2-2} and \eqref{eq2-1}. Here, the auxiliary parameter $\widetilde{U}^\mu$ is always given by the following explicit formula:
    \begin{align*}
	\widetilde{U}^\mu=\frac{1}{Z}\sum_{i=1}^4\nu_in_iU_i^\mu,
	\end{align*}
    and the auxiliary parameters $\widetilde{\mu}_i$ $(i=2,3,4)$ are determined by the relation
    \begin{align*}
	\widetilde{M}_iA_i&=\int_{\mathbb{R}^3}f_i\frac{dp_i}{p_i^0}\pm\frac{\nu_1}{\nu_i}\left(\widetilde{M}_1A_1-\int_{\mathbb{R}^3}f_1\frac{dp_1}{p_1^0}\right),\quad (+\ \text{for}\ i=2,\ \text{and}\ -\ \text{for}\ i=3,4).
\end{align*}
once $\widetilde{\beta}$ and $\widetilde{\mu}_1$ are given. The remaining parameters $\widetilde{\beta}$ and $\widetilde{\mu}_1$ are determined according to the following two cases: 
\begin{enumerate}
\item (Generic Case: $\xi(\beta^\#)\neq0$) A unique  solution pair exists. Specifically, $\widetilde{\beta}$ is determined as the unique solution to \begin{align*}
	\frac{\widetilde{M}_3\widetilde{M}_4}{\widetilde{M}_1\widetilde{M}_2}\cdot\frac{\left(\int_{\mathbb{R}^3}f_1\frac{dp_1}{p_1^0}+\frac{1}{\nu_1}\frac{Z-\Sigma}{\xi}\right)\left(\int_{\mathbb{R}^3}f_2\frac{dp_2}{p_2^0}+\frac{1}{\nu_2}\frac{Z-\Sigma}{\xi}\right)}{\left(\int_{\mathbb{R}^3}f_3\frac{dp_3}{p_3^0}-\frac{1}{\nu_3}\frac{Z-\Sigma}{\xi}\right)\left(\int_{\mathbb{R}^3}f_4\frac{dp_4}{p_4^0}-\frac{1}{\nu_4}\frac{Z-\Sigma}{\xi}\right)}=\frac{g_{s_1}g_{s_2}}{g_{s_3}g_{s_4}},
\end{align*} on the domain $\mathcal{D}_\beta$, and $\widetilde{\mu}_1$ is given by the relations:
\begin{align*}
\widetilde{M}_1A_1&=\int_{\mathbb{R}^3}f_1\frac{dp_1}{p_1^0}+\frac{1}{\nu_1}\frac{Z-\Sigma}{\xi}.
\end{align*} 

\item (Non-generic Case: $\xi(\beta^\#)=0$) At least one solution pair exists but may not be unique. A valid solution pair can be found by setting $\widetilde{\beta}=\beta^\#$ and solving the following equation with respect to $\widetilde{\mu}_1$:
\begin{align*}
	\frac{\widetilde{M}_3\widetilde{M}_4}{\widetilde{M}_1\widetilde{M}_2}\frac{\widetilde{M}_1A_1\left(\int_{\mathbb{R}^3}f_2\frac{dp_2}{p^0_2}+\frac{\nu_1}{\nu_2}\left(\widetilde{M}_1A_1-\int_{\mathbb{R}^3}f_1\frac{dp_1}{p_1^0}\right)\right)}{\left(\int_{\mathbb{R}^3}f_3\frac{dp_3}{p^0_3}-\frac{\nu_1}{\nu_3}\left(\widetilde{M}_1A_1-\int_{\mathbb{R}^3}f_1\frac{dp_1}{p_1^0}\right)\right)\left(\int_{\mathbb{R}^3}f_4\frac{dp_4}{p^0_4}-\frac{\nu_1}{\nu_4}\left(\widetilde{M}_1A_1-\int_{\mathbb{R}^3}f_1\frac{dp_1}{p_1^0}\right)\right)}=\frac{g_{s_1}g_{s_2}}{g_{s_3}g_{s_4}},
\end{align*}
on the domain $\mathcal{D}_\mu$.
\end{enumerate}
\end{theorem}
\vspace{3mm}
The following theorem establishes the H-theorem for our model presented in Theorem \ref{JF3}. Specifically, we prove that the proposed relaxation operator satisfies the crucial inequality \eqref{htheorem}, ensuring consistency with the relativistic Boltzmann equation \eqref{Boltzmann}.
\begin{theorem}[H-theorem]\label{htheorem2}
	The relaxation operator proposed in Theorem \ref{JF3} satisfies the  inequality
	\begin{align*}
		\sum_{i=1}^4\int_{\mathbb{R}^3}\widetilde{C}_i\ln \left(\frac{h^3f_i}{g_{s_i}}\right) dp_i\leq 0
	\end{align*}
where the equality holds if and only if $f_i=\mathcal{J}_i$ for $i=1,\cdots,4$.
	\begin{proof} 
	By the definition, we have
	\begin{align*}
		\sum_{i=1}^4\int_{\mathbb{R}^3}\widetilde{C}_i\ln \left(\frac{h^3\mathcal{J}_i}{g_{s_i}}\right) dp_i&=\sum_{i=1}^4\int_{\mathbb{R}^3}\widetilde{C}_i(\widetilde{\beta}\widetilde{\mu}_i-\widetilde{\beta}\widetilde{U}^\mu p_{i\mu})\ dp_i\\
		&=\widetilde{\beta}\sum_{i=1}^4\widetilde{\mu}_i\int_{\mathbb{R}^3}\widetilde{C}_i\ dp_i+\widetilde{\beta}\widetilde{U}^\mu\sum_{i=1}^4\int_{\mathbb{R}^3}p_{i\mu}\widetilde{C}_i\ dp_i\\
		&=\widetilde{\beta}\sum_{i=1}^4\widetilde{\mu}_i\int_{\mathbb{R}^3}\widetilde{C}_i\ dp_i,
	\end{align*}
	where we used \eqref{eq2-1}. It follows from  	\eqref{eq2-2} and \eqref{eq2-1} that
	\begin{align*}
		\sum_{i=1}^4\widetilde{\mu}_i\int_{\mathbb{R}^3}\widetilde{C}_i\ dp_i=&\widetilde{\mu}_1\int_{\mathbb{R}^3}\widetilde{C}_1\ dp_1+\widetilde{\mu}_2\int_{\mathbb{R}^3}\widetilde{C}_1\ dp_1-\widetilde{\mu}_3\int_{\mathbb{R}^3}\widetilde{C}_1\ dp_1-\widetilde{\mu}_4\int_{\mathbb{R}^3}\widetilde{C}_1\ dp_1   \\
		=&(\widetilde{\mu}_1+\widetilde{\mu}_2-\widetilde{\mu}_3-\widetilde{\mu}_4)\int_{\mathbb{R}^3}\widetilde{C}_1\ dp_1\\
		=&0  ,
	\end{align*}
yielding
\begin{align*}
	\sum_{i=1}^4\int_{\mathbb{R}^3}\widetilde{C}_i\ln \left(\frac{h^3\mathcal{J}_i}{g_{s_i}}\right) dp_i=0.
\end{align*}
Therefore, we get
	\begin{align*}
		\sum_{i=1}^4\int_{\mathbb{R}^3}\widetilde{C}_i\ln \left(\frac{h^3f_i}{g_{s_i}}\right) dp_i&=\sum_{i=1}^4\int_{\mathbb{R}^3}\widetilde{C}_i\left\{\ln \left(\frac{h^3f_i}{g_{s_i}}\right) -\ln \left(\frac{h^3\mathcal{J}_i}{g_{s_i}}\right)\right\}\,dp_i\\
		&=\sum_{i=1}^4\int_{\mathbb{R}^3}\nu_i(\mathcal{J}_i-f_i)\ln \left(\frac{f_i}{\mathcal{J}_i}\right)dp_i\\
		&\leq 0,
	\end{align*}
	where we used the elementary inequality $(x-y)\ln(x/y)\leq0$ in the last line. Then, the equality holds only if $\mathcal{J}_i\equiv f_i$, which completes the proof. 
	\end{proof}
\end{theorem} 

\section{Numerical Simulation}\label{sec-numeric}
In this section, we numerically verify the theoretical properties of the proposed model by checking the energy and momentum conservation, as well as the entropy behavior predicted by the H-theorem. For simplicity, we consider a spatially homogeneous system with one-dimensional momentum, i.e.,  $p_i^\mu=(\sqrt{1+|p_i^1|^2},p_i^1)$, where $p_i^1$ denotes the one-dimensional momentum component. Under this assumption, the governing equation reduces to the following simplified model:
\begin{align}\label{simplified}
	\partial_tf_i=\frac{\nu_i}{p^0}(\mathcal{J}_i-f_i), \qquad i=1,\cdots,4,
\end{align}
where $\mathcal{J}_i$ is the J\"{u}ttner distribution, given in Theorem \ref{JF3}. Note that by \eqref{eq2-1} the time evolution of the momentum distribution functions $f_i$ governed by the simplified model \eqref{simplified} satisfies
\begin{align}
    \label{conservation1}
		\frac{d}{dt}\left(\int_{ \mathbb{R}^1} f_idp_i + \int_{ \mathbb{R}^1} f_jdp_j\right)&=0,\qquad \text{for}\ (i,j)=(1,3),(1,4),(2,4)\\
    \label{conservation2}
		\frac{d}{dt}\left(\sum_{i=1}^4\int_{ \mathbb{R}^1}p_i^0f_idp_i\right)&=0,\quad 
		\frac{d}{dt}\left(\sum_{i=1}^4\int_{ \mathbb{R}^1}p_i^1f_idp_i\right)=0,
\end{align}
and by Theorem \ref{htheorem2} we have
\begin{align}\label{H thm num}
	\frac{d}{dt}\left(\sum_{i=1}^4\int_{ \mathbb{R}^1}f_i\ln{\left(\frac{h^3f_i}{g_{s_i}}\right)}dp_i\right)\leq0.
\end{align}
In the following tests,
we update numerical solutions using the classical fourth-order Runge-Kutta method. To calculate $\widetilde{\beta}$ defined as the unique solution to \eqref{mu_rel2}, we use Newton's method with initial point $\beta^\#$.

\subsection{Case 1. initial data near equilibrium} We first consider the case when initial distribution functions take the form of 
Maxwell-J\"{u}ttner distributions, i.e., for each species we set
\begin{align*}
	f_i^0(p^\mu_i)=\frac{g_{s_i}}{h^3}\exp\left(\beta_i^0\mu_i^0-\beta_i^0U_i^\mu p_{i\nu}\right),\quad i=1,...,4,
\end{align*}
where we used the following macroscopic quantities:
\begin{align*}
	\mu_1=1.8,\ \mu_2=1.3,\ \mu_3=1.0,\ \mu_4=1.0,\\
	U_1=0.5,\ U_2=-0.3,\ U_3=1.0,\ U_4=0.2,\\
	\beta_1=0.8,\ \beta_2=1.1,\ \beta_3=0.9,\ \beta_4=1.2.
\end{align*}
Also, for the collision frequencies $\nu_i$, we choose the following values:
\begin{align*}
	\nu_1=3,\ \nu_2=2,\ \nu_3=1, \nu_4=4,
\end{align*}
and we fix the masses of gas particles:
\begin{align*}
	m_1=2.0,\ m_2=1.0,\ m_3=3,\ m_4=1.
\end{align*}
% We report numerical results in Figures \ref{fig 1}-\ref{fig 5}.\\
 \begin{figure}
 	\centering
 	\includegraphics[width=10cm]{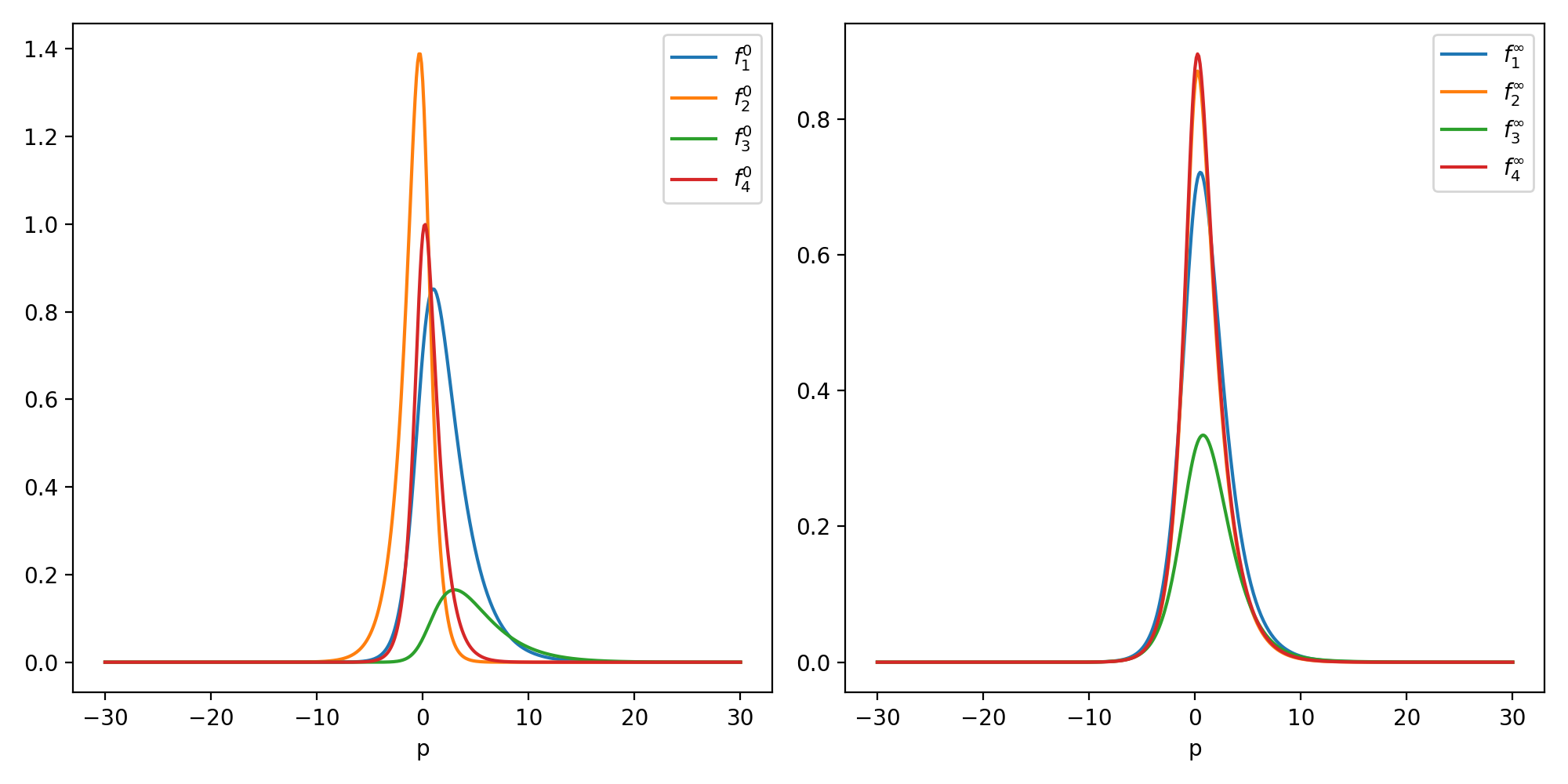}
 	\caption{(Left) Distributions $f_i$ at time $t=0$, (Right) Distributions $f_i$ at time $t=10$.}\label{fig 1}
    	\centering
 \end{figure}
\begin{figure}
\includegraphics[width=7cm]{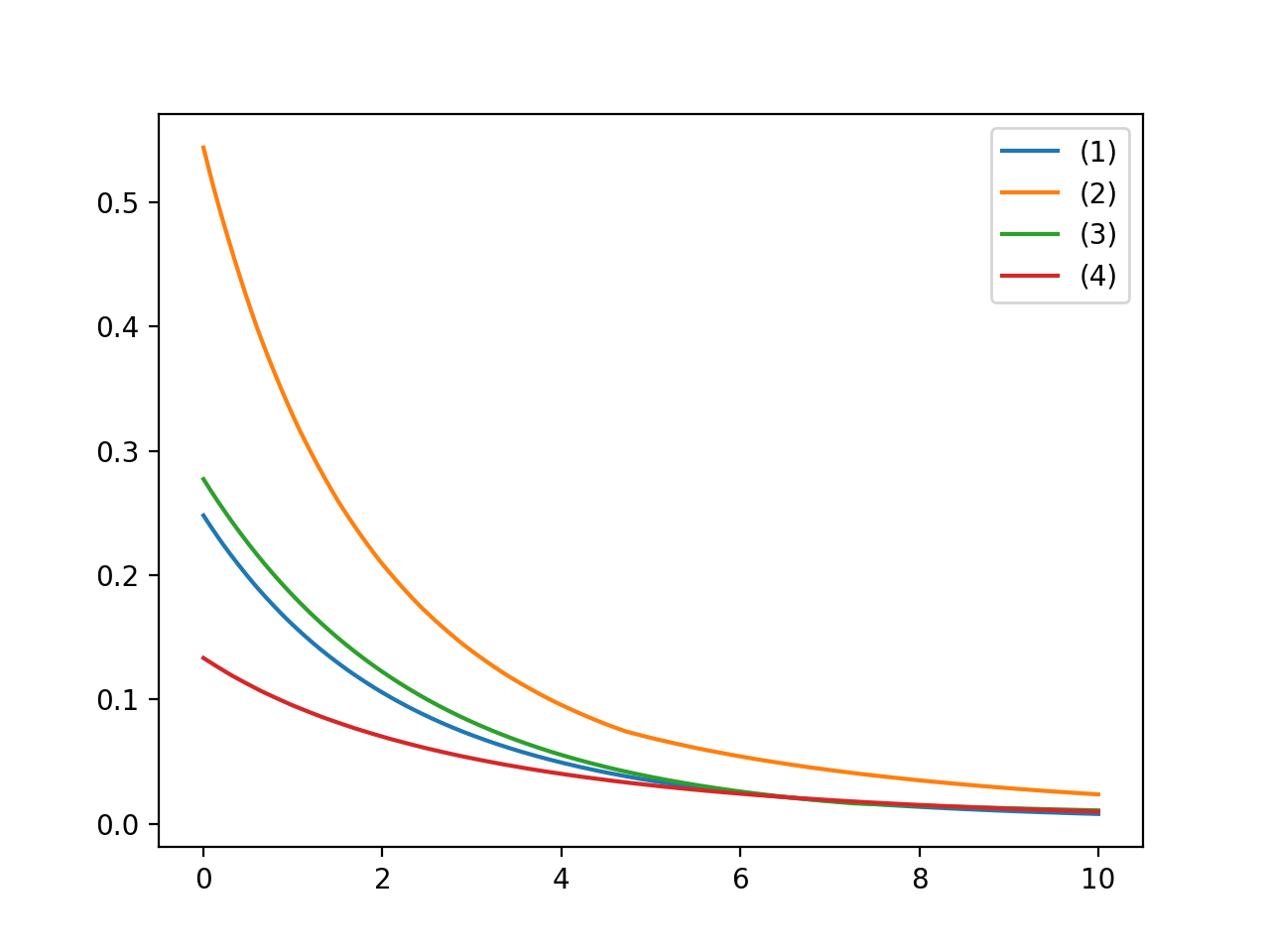}
	\caption{Time evolution of  $\|f_i(t,\cdot)-\mathcal{J}_i(t,\cdot)\|_\infty$, $i=1,\,2,\,3,\,4$.}\label{fig 2}
    \end{figure}
In Figure \ref{fig 1}, we present the configuration of the momentum distributions both at the initial time and the large time.
% which shows that momentum distribution functions seem to evolve to have a common velocity. 
Indeed, in Figure \ref{fig 2}, 
the time evolution of $\|\cdot\|_\infty$-norm between $f_i$ and J\"{u}ttner-type distribution $\mathcal{J}_i$ implies that each $f_i$ converges to $\mathcal{J}_i$ as time goes.
\begin{figure}
	\centering
    % \begin{subfigure}[h]{0.49\linewidth}
    %     \end{subfigure}
\includegraphics[width=9cm]{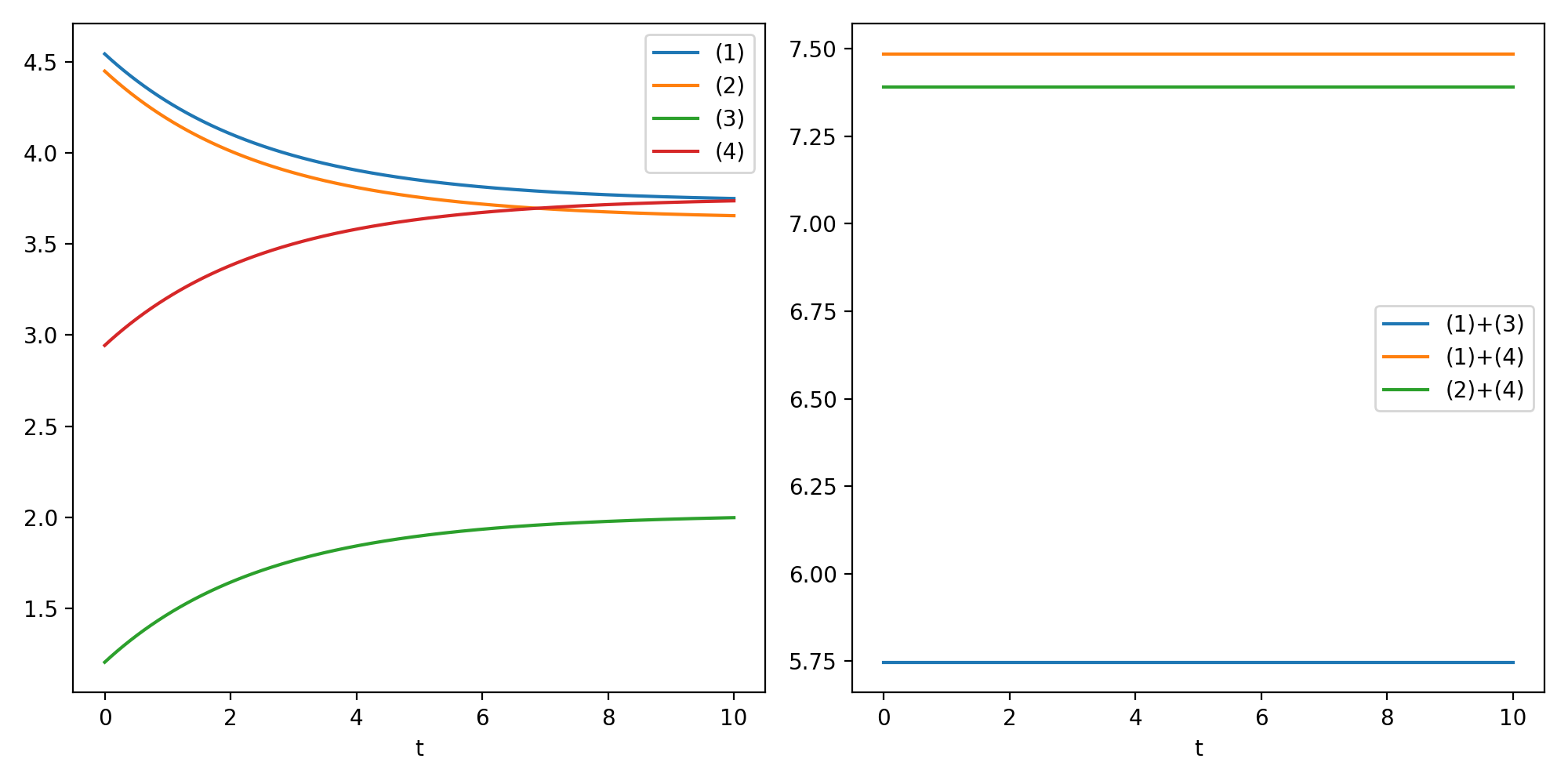}
	\caption{(Left) Behavior of $\int_{\mathbb{R}^1}f_idp_i$ over time $t$, (Right) Behavior of $\int_{\mathbb{R}^1}f_idp_i+\int_{\mathbb{R}^1}f_jdp_j$ over time $t$.}\label{fig 3}
\end{figure}
\begin{figure}
\includegraphics[width=9cm]{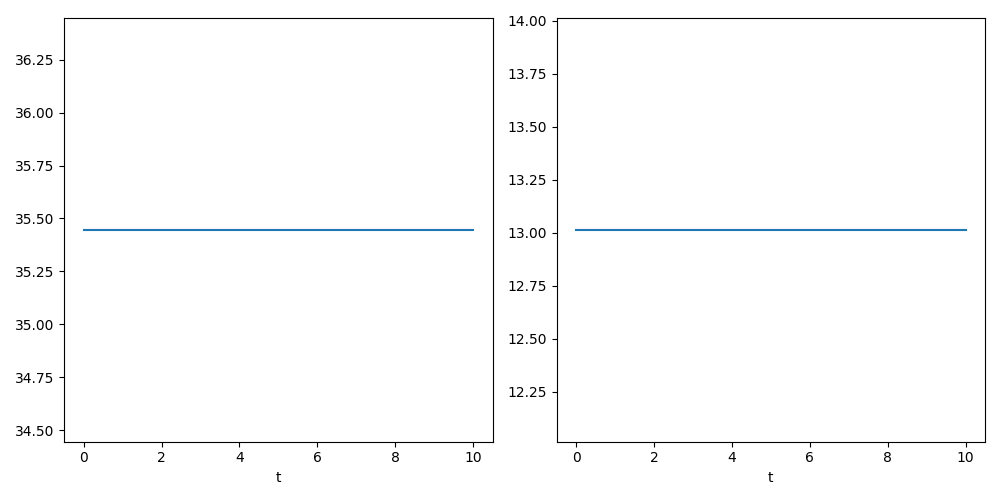}
	\caption{(Left) Behavior of $\sum_{i=1}^4\int_{\mathbb{R}^1}p_i^0f_idp_i$ over time $t$, (Right) Behavior of $\sum_{i=1}^4\int_{\mathbb{R}^1}p_i^1f_idp_i$ over time $t$}\label{fig 4}
\includegraphics[width=7cm]{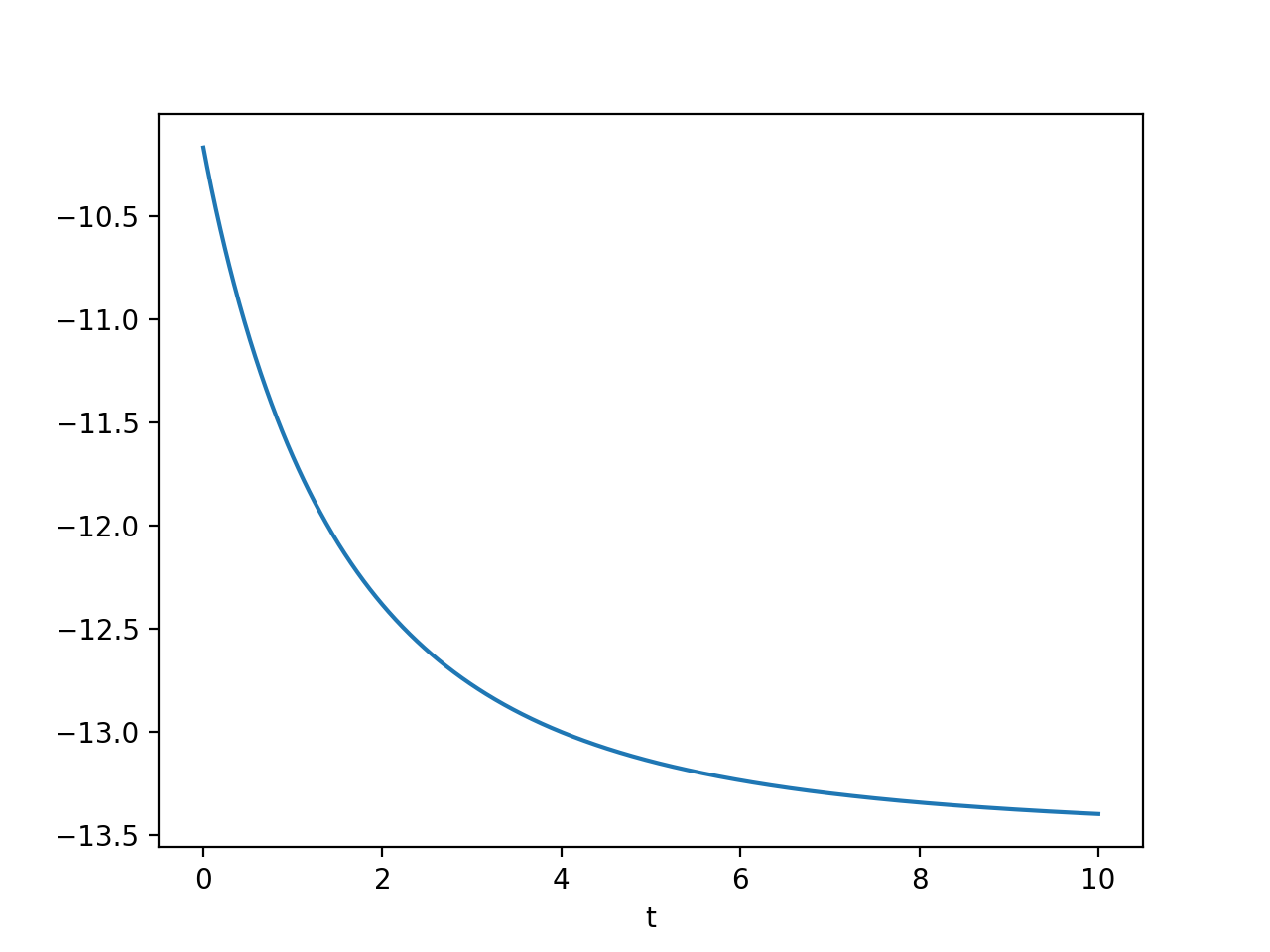}
	\caption{Behavior of $\sum_{i=1}^4\int_{\mathbb{R}^1}f_i\ln{\left(\frac{h^3f_i}{g_{s_i}}\right)}dp_i$ over time $t$.}\label{fig 5}
\end{figure}
In Figure \ref{fig 3}-\ref{fig 4}, the numerical results show that macroscopic quantities are preserved in time as expected in \eqref{conservation1} and \eqref{conservation2}. In Figure \ref{fig 5}, we can observe the expected monotonic behavior of entropy in \eqref{H thm num}.

\subsection{Case 2. initial data far from equilibrium} In this second case, we consider a different set of initial conditions where the distributions have a compactly supported, isosceles triangle shape, similar to the setup used in \cite{martalo2024investigating}. The goal here is to investigate how these non-equilibrium, highly localized initial distributions evolve towards equilibrium. Specifically, each species' initial momentum distribution function, $f_i^0$, is supported on the following intervals: \begin{align*} \text{sppt}(f_1^0)=[-9,-2], \quad \text{sppt}(f_2^0)=[-7.5,3], \quad \text{sppt}(f_3^0)=[-3,1], \quad \text{sppt}(f_4^0)=[-5.5,-5]. \end{align*} The heights of these initial distributions are chosen as: \begin{align*} \max{(f_1^0)}=0.2, \quad \max{(f_2^0)}=0.38, \quad \max{(f_3^0)}=0.25, \quad \max{(f_4^0)}=0.28. \end{align*}
For the test, we fix the collision frequencies $\nu_i$ and the masses $m_i$ as in Case 1 for consistency. 

The left panel in Figure \ref{fig 6} shows the initial triangular distributions of the four species, providing a stark contrast to the Maxwell-J\"{u}ttner distribution used in Case 1. This setup represents a highly non-equilibrium state, where the initial distributions are localized. At time $t=30$, we note that each species distribution function has the form of Maxwell-J\"{u}ttner distribution. The relaxation towards equilibrium is demonstrated by the time evolution of distance between $f_i$ and $\mathcal{J}_i$ in Figure \ref{fig 7}.
\begin{figure} \centering 
\includegraphics[width=10cm]{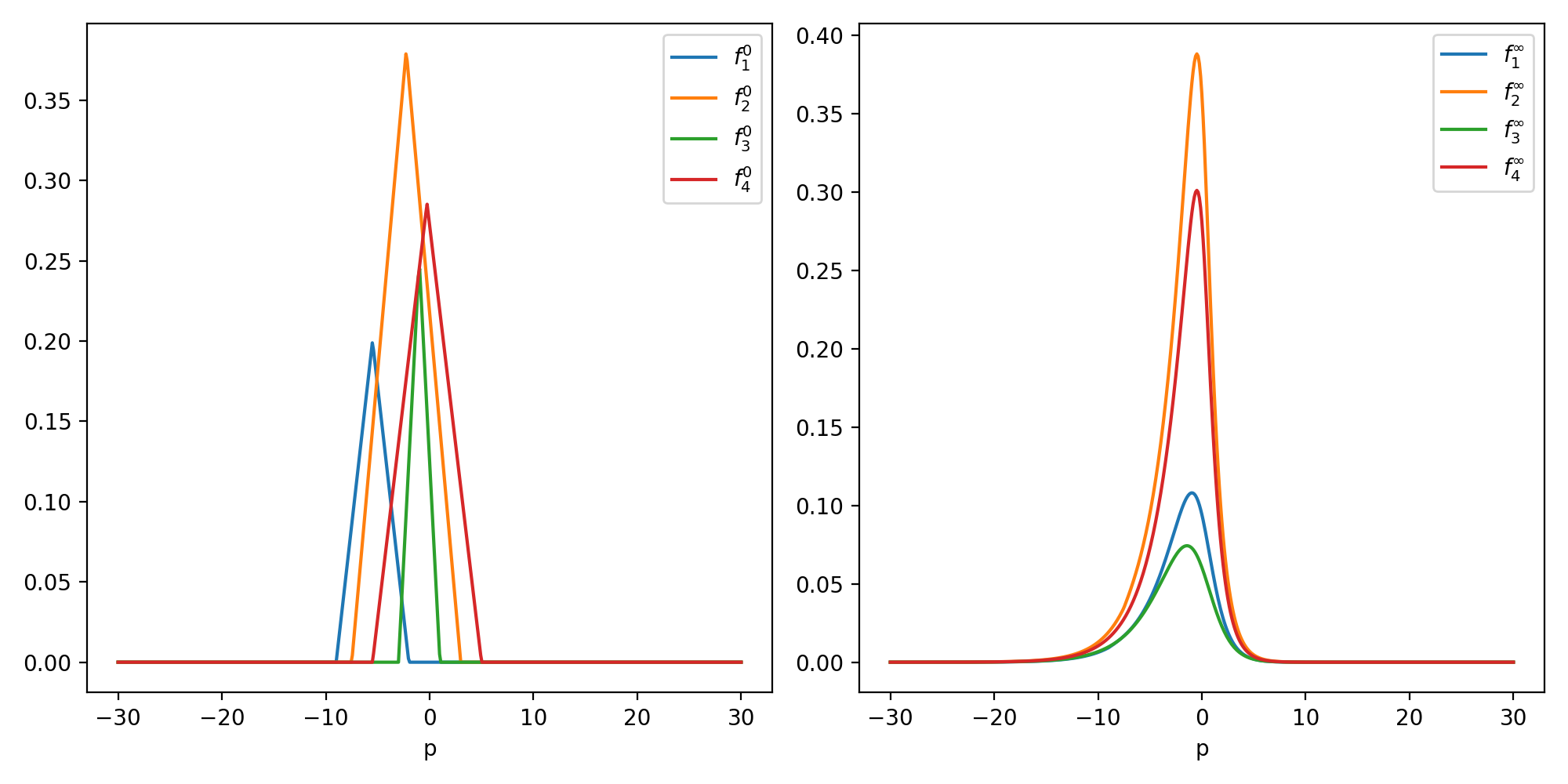} \caption{(Left) Distributions $f_i$ at time $t=0$, (Right) Distributions $f_i$ at time $t=30$.}\label{fig 6} \includegraphics[width=7cm]{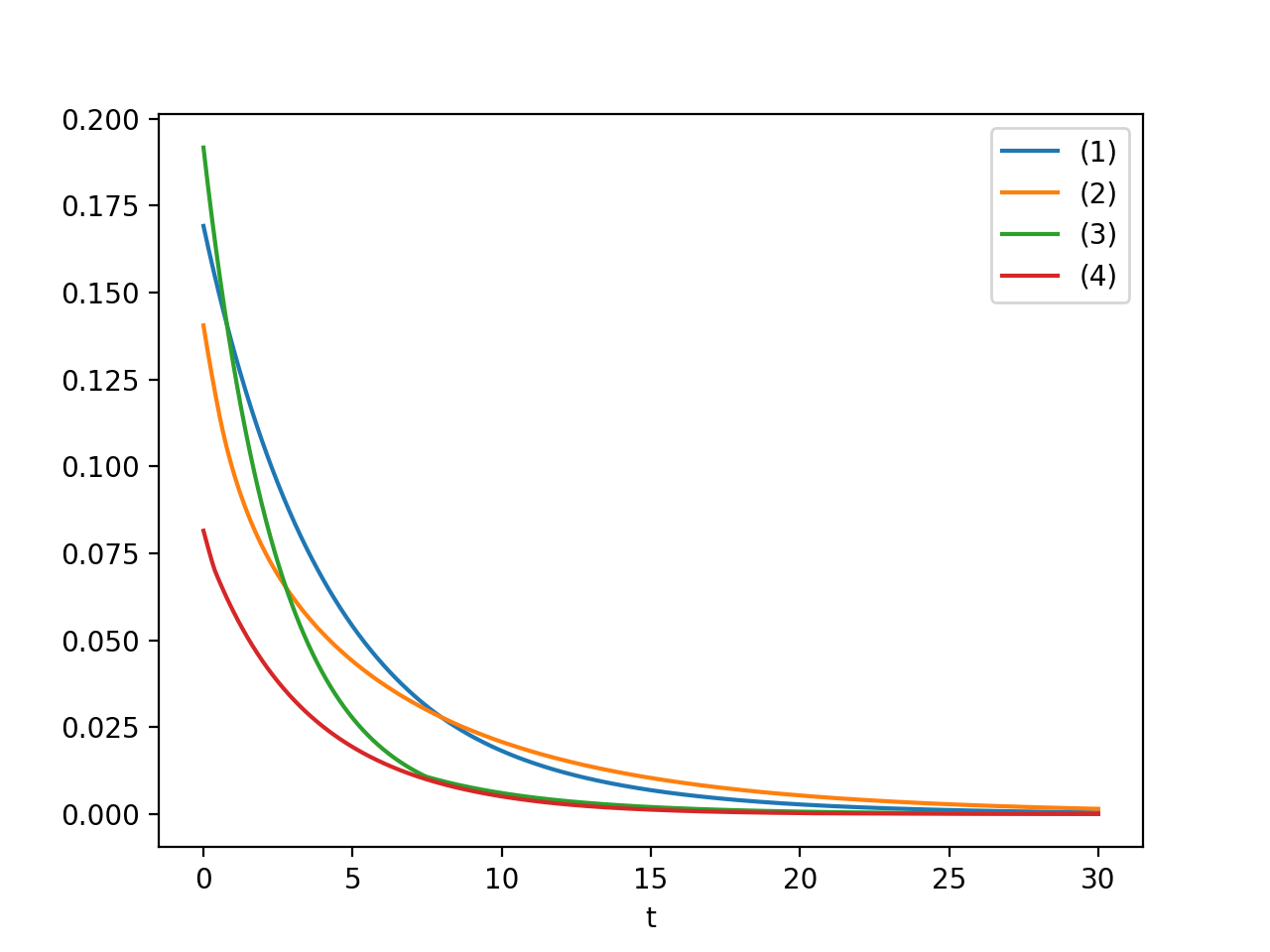} \caption{Behavior of $\|f_i(t,\cdot)-\mathcal{J}_i(t,\cdot)\|_\infty$ over time $t$.}\label{fig 7} 
\end{figure}

\begin{figure} \centering 
\includegraphics[width=10cm]{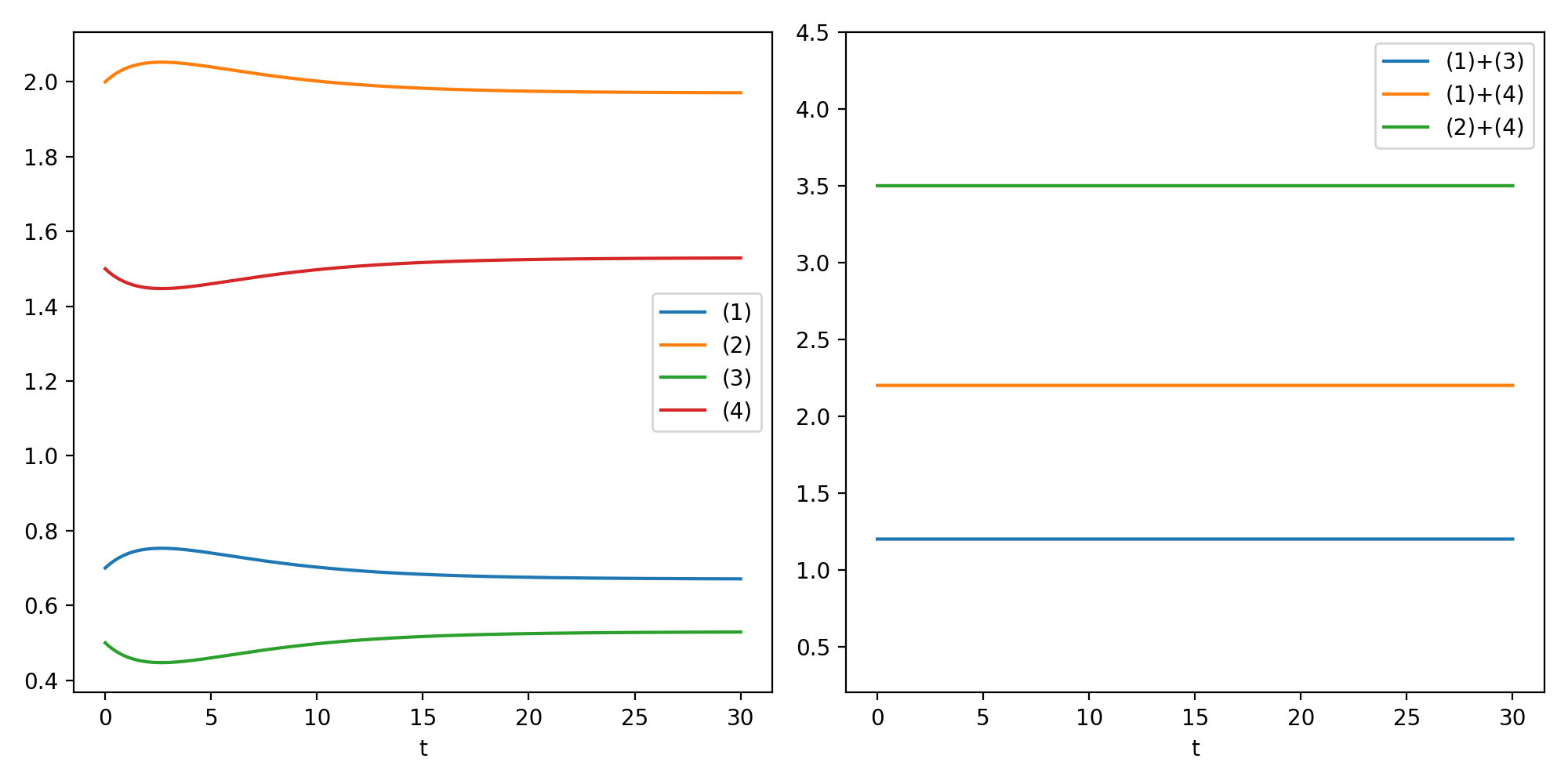} \caption{(Left) Behavior of $\int_{\mathbb{R}^1}f_idp_i$ over time $t$, (Right) Behavior of $\int_{\mathbb{R}^1}f_idp_i+\int_{\mathbb{R}^1}f_jdp_j$ over time $t$.}\label{fig 8} \includegraphics[width=10cm]{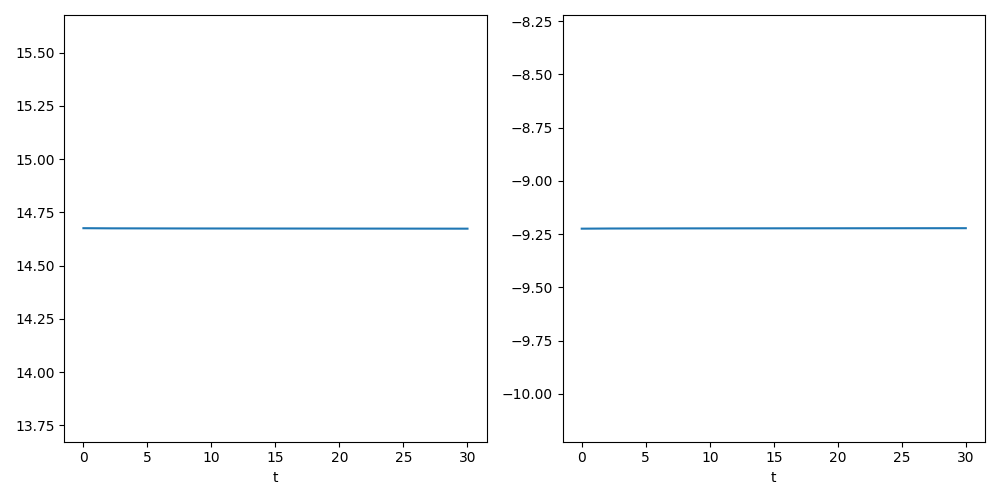} \caption{(Left) Behavior of $\sum_{i=1}^4\int_{\mathbb{R}^1}p_i^0f_idp_i$ over time $t$, (Right) Behavior of $\sum_{i=1}^4\int_{\mathbb{R}^1}p_i^1f_idp_i$ over time $t$}\label{fig 9}
\end{figure}

\begin{figure}\centering \includegraphics[width=7cm]{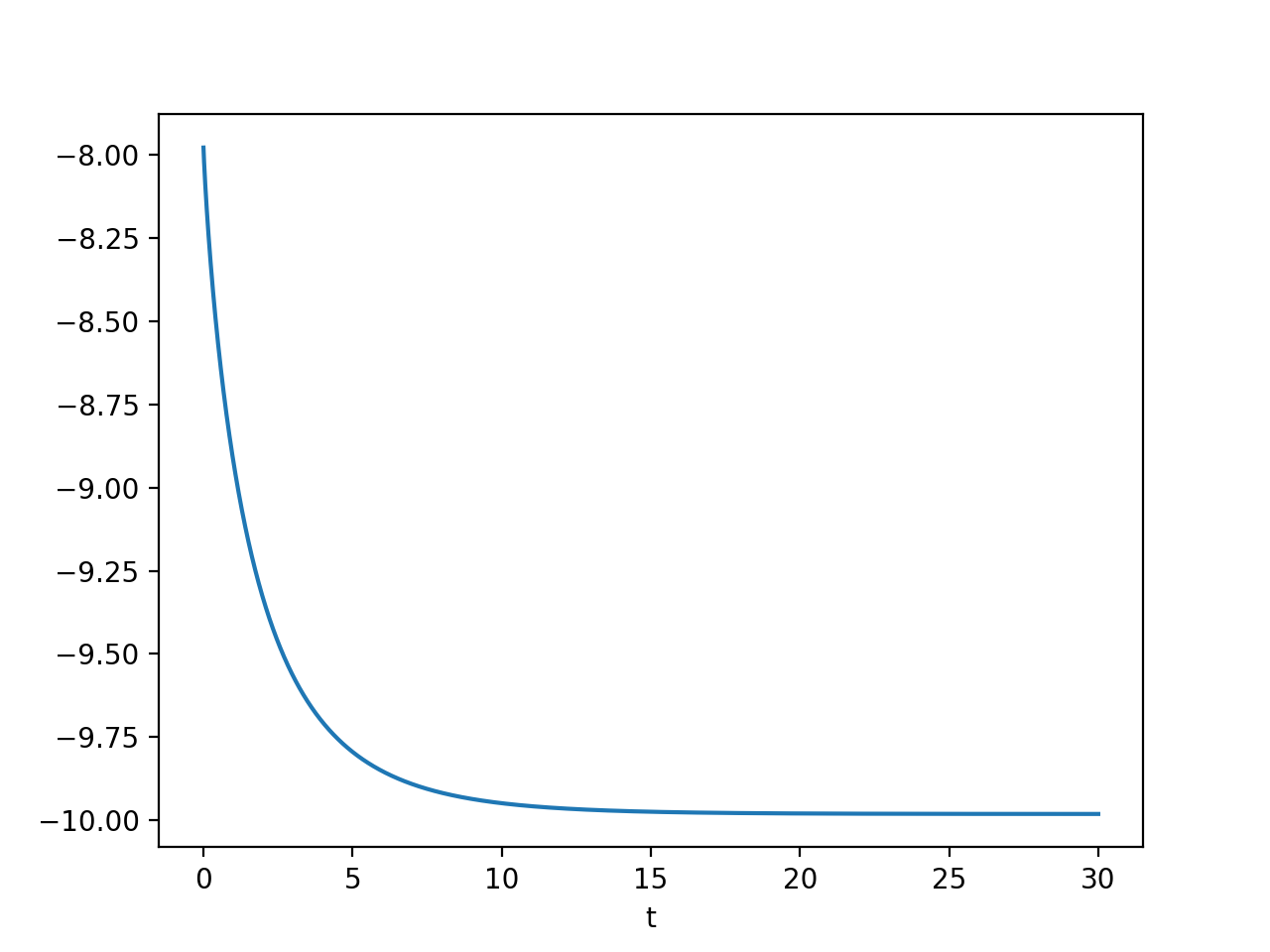} \caption{Behavior of $\sum_{i=1}^4\int_{\mathbb{R}^1}f_i\ln{\left(\frac{h^3f_i}{g_{s_i}}\right)}dp_i$ over time $t$.}\label{fig 10}\end{figure}

In Figures \ref{fig 8}-\ref{fig 10}, we confirm that the conservation of mass, momentum and energy, and the monotonic behavior of entropy are not affected by initial data.

\section*{Acknowledgments}
S.-Y. Cho was supported by Learning \& Academic research institution for Master’s·PhD students, and
Postdocs (LAMP) Program of the National Research Foundation of Korea (NRF) grant funded by the
Ministry of Education (No. RS-2023-00301974). B.-H. Hwang was supported by Basic Science Research Program through the National Research Foundation of Korea(NRF) funded by the Ministry of Education(No. NRF-2019R1A6A1A10073079 and RS-2024-00462755). M.-S. Lee was supported by Basic Science Research Programs through the National Research Foundation of Korea (NRF) funded by the Ministry of Education (RS-2023-00244475 and RS-2024-00462755). S.-B. Yun was supported by the National Research Foundation of Korea(NRF) grant funded by
the Korean goverment(MSIT) (RS-2023-NR076676).

\bibliographystyle{unsrt}
\bibliography{library}

\end{document}